\documentclass[12pt,a4paper]{article}
\usepackage{amssymb}
\usepackage{amsmath, amscd}
\usepackage{amsthm}
\usepackage{mathtools}
\usepackage{mathrsfs}
\usepackage{booktabs}
\usepackage{perpage}
\usepackage[all,cmtip]{xy}
\usepackage{setspace}
\usepackage{amsbsy}
\usepackage{hyperref}
\usepackage[T1]{fontenc}
\usepackage{verbatim}
\usepackage{mathdots}
\usepackage{eucal}
\usepackage{mathrsfs}
\usepackage{ifthen}
\usepackage{blkarray}
\usepackage{multirow}
\usepackage{enumerate}
\usepackage[dvipsnames]{xcolor}
\usepackage{tikz}
\usetikzlibrary{matrix,positioning,decorations.pathreplacing}
\usepackage{tikz-cd}
\usetikzlibrary{matrix,positioning,decorations.pathreplacing}
\usepackage{fullpage}
\usepackage{colonequals}
\usepackage{braket}
\usepackage{lipsum}
\usepackage{float}
\usepackage{multirow,bigdelim}

\usetikzlibrary{matrix}
\usetikzlibrary{calc}

\usepackage{graphicx}

\makeatletter
\newcommand{\sigmaop}[1]{\mathop{\mathpalette\@sigmaop{#1}}\slimits@}
\newcommand{\@sigmaop}[2]{%
  \vphantom{\sum}%
  \sbox\z@{$\m@th#1\sum$}%
  \dimen@=\ht\z@ \advance\dimen@\dp\z@
  \dimen\tw@=\wd\z@
  \ifx#1\displaystyle\dimen@=.9\dimen@\fi
  \ooalign{%
    \hidewidth
    $\vcenter{\hbox{$\m@th#1#2$\kern.3\dimen\tw@}%
     \ifx#1\scriptstyle\kern-.25ex\fi}$\hidewidth\cr
    $\vcenter{\hbox{%
      \resizebox{!}{\dimen@}{$\m@th\boxtimes$}%
    }\ifx#1\scriptstyle\kern-.25ex\fi}$\cr
  }%
}
\makeatother

\mathtoolsset{showonlyrefs=true}

\numberwithin{equation}{subsection}

\newtheorem{theorem}{Theorem}[subsection]
\newtheorem{lemma}[theorem]{Lemma}

\newtheorem{conjecture}[theorem]{Conjecture}

\newtheorem{corollary}[theorem]{Corollary}
\newtheorem{definition}[theorem]{Definition}

\newtheorem{proposition}[theorem]{Proposition}

\newtheorem{assumption}[theorem]{Assumption}

\newtheorem*{thm1}{Theorem 1}
\newtheorem*{thm2}{Theorem 2}
\newtheorem*{thm3}{Theorem 3}

\newtheorem*{conj1}{Conjecture 1}
\newtheorem*{conj2}{Conjecture 2}

\theoremstyle{remark}

\newcommand{\GZip}{\mathop{\text{$G$-{\tt Zip}}}\nolimits}

\newcommand{\HZip}{\mathop{\text{$H$-{\tt Zip}}}\nolimits}

\newcommand{\GF}{\mathop{\text{$G$-{\tt ZipFlag}}}\nolimits}

\newcommand{\HF}{\mathop{\text{$H$-{\tt ZipFlag}}}\nolimits}

\newskip\procskipamount
\newskip\interskipamount
\newskip\refskipamount
\newcommand{\procskip}{\vskip\procskipamount}
\newcommand{\interskip}{\vskip\interskipamount}
\newcommand{\refskip}{\vskip\refskipamount}

\newcommand{\procbreak}{\par
   \ifdim\lastskip<\procskipamount\removelastskip
   \penalty-100
   \procskip\fi
   \noindent\ignorespaces}

\newcommand{\titlebreak}{\par%
\ifdim\lastskip<\interskipamount\removelastskip%
\penalty10000%
\interskip\fi%
\noindent}%

\newcommand{\interbreak}{\par%
\ifdim\lastskip<\interskipamount\removelastskip%
\penalty-100%
\interskip\fi%
\noindent\ignorespaces}%

\newcommand{\refbreak}{\par%
\ifdim\lastskip<\refskipamount\removelastskip%
\penalty-100%
\refskip\fi%
\noindent\ignorespaces}%

\newcounter{listcounter}
\newcounter{deflistcounter}
\newcounter{equivcounter}

\newskip{\itemsepamount}
\newskip{\topsepamount}
\newenvironment{assertionlist}{%
  \begin{list}
    {\upshape (\arabic{listcounter})}
    {\setlength{\leftmargin}{18pt}
     \setlength{\rightmargin}{0pt}
     \setlength{\itemindent}{0pt}
     \setlength{\labelsep}{5pt}
     \setlength{\labelwidth}{13pt}
     \setlength{\listparindent}{\parindent}
     \setlength{\parsep}{0pt}
     \setlength{\itemsep}{\itemsepamount}
     \setlength{\topsep}{\topsepamount}
     \usecounter{listcounter}}}
  {\end{list}}

\newenvironment{definitionlist}{%
  \begin{list}
    {\upshape (\alph{deflistcounter})}
    {\setlength{\leftmargin}{18pt}
     \setlength{\rightmargin}{0pt}
     \setlength{\itemindent}{0pt}
     \setlength{\labelsep}{5pt}
     \setlength{\labelwidth}{13pt}
     \setlength{\listparindent}{\parindent}
     \setlength{\parsep}{0pt}
     \setlength{\itemsep}{\itemsepamount}
     \setlength{\topsep}{\topsepamount}
     \usecounter{deflistcounter}}}
  {\end{list}}

\newenvironment{equivlist}{%
  \begin{list}
    {\upshape (\roman{equivcounter})}
    {\setlength{\leftmargin}{18pt}
     \setlength{\rightmargin}{0pt}
     \setlength{\itemindent}{0pt}
     \setlength{\labelsep}{5pt}
     \setlength{\labelwidth}{13pt}
     \setlength{\listparindent}{\parindent}
     \setlength{\parsep}{0pt}
     \setlength{\itemsep}{\itemsepamount}
     \setlength{\topsep}{\topsepamount}
     \usecounter{equivcounter}}}
  {\end{list}}

\newenvironment{bulletlist}{%
  \begin{list}
    {\upshape \textbullet}
    {\setlength{\leftmargin}{18pt}
     \setlength{\rightmargin}{0pt}
     \setlength{\itemindent}{0pt}
     \setlength{\labelsep}{6pt}
     \setlength{\labelwidth}{12pt}
     \setlength{\listparindent}{\parindent}
     \setlength{\parsep}{0pt}
     \setlength{\itemsep}{\itemsepamount}
     \setlength{\topsep}{\topsepamount}}}
  {\end{list}}

% ...

\newcommand{\Ccal}{{\mathcal C}}

\newcommand{\Ecal}{{\mathcal E}}
\newcommand{\Fcal}{{\mathcal F}}

\newcommand{\Hcal}{{\mathcal H}}

\newcommand{\Mcal}{{\mathcal M}}
\newcommand{\Ncal}{{\mathcal N}}
\newcommand{\Ocal}{{\mathcal O}}

\newcommand{\Ucal}{{\mathcal U}}
\newcommand{\Vcal}{{\mathcal V}}

\newcommand{\Xcal}{{\mathcal X}}

\newcommand{\Zcal}{{\mathcal Z}}

\newcommand{\mfr}{{\mathfrak m}}

\newcommand{\pfr}{{\mathfrak p}}

\newcommand{\Sfr}{{\mathfrak S}}

\renewcommand{\AA}{\mathbb{A}}
\newcommand{\BB}{\mathbb{B}}
\newcommand{\CC}{\mathbb{C}}

\newcommand{\FF}{\mathbb{F}}
\newcommand{\GG}{\mathbb{G}}

\newcommand{\NN}{\mathbb{N}}

\newcommand{\QQ}{\mathbb{Q}}

\newcommand{\TT}{\mathbb{T}}

\newcommand{\ZZ}{\mathbb{Z}}

\newcommand{\Ascr}{{\mathscr A}}
\newcommand{\Bscr}{{\mathscr B}}

\newcommand{\Hscr}{{\mathscr H}}

\newcommand{\Lscr}{{\mathscr L}}

\newcommand{\Pscr}{{\mathscr P}}

\newcommand{\Sscr}{{\mathscr S}}

\newcommand{\cent}{{\rm Cent}}

\DeclareMathOperator{\Gal}{Gal}

\DeclareMathOperator{\Span}{Span}

\DeclareMathOperator{\Lie}{Lie}

\DeclareMathOperator{\ord}{ord}

\DeclareMathOperator{\pr}{pr}

\DeclareMathOperator{\Sbt}{Sbt}
\DeclareMathOperator{\Sh}{Sh}
\DeclareMathOperator{\spec}{Spec}

\DeclareMathOperator{\Sch}{Sbt}

\DeclareMathOperator{\Tr}{Tr}

\newcommand{\Hasse}{\mathsf{pHa}}
\newcommand{\zip}{\mathsf{zip}}
\newcommand{\GS}{\mathsf{GS}}
\newcommand{\lw}{\mathsf{lw}}
\newcommand{\hw}{\mathsf{hw}}
\newcommand{\flag}{\mathsf{flag}}

\DeclareMathOperator{\GL}{GL}

\DeclareMathOperator{\GSp}{GSp}

\DeclareMathOperator{\Sp}{Sp}

\DeclareMathOperator{\GU}{GU}

\newcommand{\gx}{(\mathbf G, \mathbf X)}

\newcommand{\End}{{\rm End}}

\newcommand{\loccit}{{\em loc.\ cit. }}

\newcommand{\diag}{{\rm diag}}

\renewcommand{\div}{{\rm div}}

\DeclareMathOperator{\Ind}{Ind}

\DeclareMathOperator{\Res}{Res}
\DeclareMathOperator{\Flag}{Flag}

\DeclareMathOperator{\ha}{ha}

\DeclareMathOperator{\Min}{Min}

\DeclareMathOperator{\Ha}{Ha}

\newcommand{\relmiddle}[1]{\mathrel{}\middle#1\mathrel{}}

\begin{document}

\author{Jean-Stefan Koskivirta}

\title{A vanishing theorem for vector-valued Siegel automorphic forms in characteristic $p$} 

\date{}

\maketitle

\begin{abstract}
We show that the space of vector-valued Siegel automorphic forms in characteristic $p$ vanishes when the weight is outside of an explicit locus. This result is a special case of a general conjecture about Hodge-type Shimura varieties formulated in previous work with W. Goldring.
\end{abstract}

\section*{Introduction}
To study automorphic forms attached to a reductive group $\textbf{G}$ over $\QQ$, it is often useful to view them as global sections of certain vector bundles on a Shimura variety. Of course, this is only possible when the group $\textbf{G}$ admits such a variety. In this paper, we investigate automorphic forms attached to the group $\textbf{G}=\GSp_{2n,\QQ}$, called Siegel automorphic forms. We are particularly interested in the properties of automorphic forms with coefficients in $\overline{\FF}_p$, for a prime number $p$. We will see that the geometric properties of the attached Siegel-type Shimura variety can be translated on the level of automorphic forms. This geometric approach was used in several previous joint papers with Wushi Goldring (\cite{Goldring-Koskivirta-Strata-Hasse}, \cite{Goldring-Koskivirta-global-sections-compositio}, \cite{Goldring-Koskivirta-divisibility}, \cite{Goldring-Koskivirta-GS-cone}).


Let $(\mathbf{G},\mathbf{X})$ be a Shimura datum of Hodge-type, in the sense of \cite{Deligne-Shimura-varieties}. In particular, $\mathbf{G}$ is a connected, reductive group over $\QQ$. If $K\subset \mathbf{G}(\AA_f)$ is a compact open subgroup, we denote by $\Sh_K(\mathbf{G},\mathbf{X})$ the attached Shimura variety of level $K$. It is an algebraic variety defined over a number field $\mathbf{E}$, called the reflex field of the Shimura datum. For example, the Siegel-type Shimura variety $\Ascr_{n,K}$ is a fundamental example of a Shimura variety, it parametrizes principally polarized abelian varieties of rank $n$ endowed with a $K$-level structure. Let $p$ be a prime of good reduction for $(\mathbf{G},\mathbf{X})$. By this we mean that the group $\mathbf{G}_{\QQ_p}$ is unramified at $p$, and that $K$ is of the form $K=K_pK^p$ with $K_p\subset \mathbf{G}(\QQ_p)$ hyperspecial and $K^p\subset \mathbf{G}(\AA_f^p)$ open compact. In particular, we can write $K_p=\mathbf{G}_{\ZZ_p}(\ZZ_p)$ for a connected reductive $\ZZ_p$-model $\mathbf{G}_{\ZZ_p}$ of $\mathbf{G}_{\QQ_p}$. By work of Kisin (\cite{Kisin-Hodge-Type-Shimura}) and Vasiu (\cite{Vasiu-Preabelian-integral-canonical-models}), $\Sh_K(\mathbf{G},\mathbf{X})$ admits a smooth canonical model $\Sscr_K$ over $\Ocal_{\mathbf{E}_\pfr}$ for any place $\pfr|p$ of $\mathbf{E}$. We denote by $S_K\colonequals \Sscr_K\otimes_{\Ocal_{\mathbf{E}_\pfr}} \overline{\FF}_p$ its special fiber.

Let $\mu\colon \GG_{\mathrm{m},\CC}\to \mathbf{G}_\CC$ be the cocharacter deduced from the Shimura datum (see section \ref{sec-GZip-HT}). It defines a parabolic subgroup $\mathbf{P}\subset \mathbf{G}_\CC$ such that the centralizer of $\mu$ is a Levi subgroup $\mathbf{L}\subset \mathbf{P}$. Choose a Borel subgroup $\mathbf{B}\subset \mathbf{P}$ and a maximal torus $\mathbf{T}\subset \mathbf{B}$. Let $\Phi^+$ denote the positive $\mathbf{T}$-roots with respect to $\mathbf{B}$, and $\Delta\subset \Phi^+$ the simple roots. Write $I\colonequals \Delta_{\mathbf{L}}\subset \Delta$ for the simple roots of $\mathbf{L}$ and $\Phi^+_{\mathbf{L}}$ for the positive roots of $\mathbf{L}$. For any character $\lambda\in X^*(\mathbf{T})$, there is an automorphic vector bundle $\Vcal_I(\lambda)$ on $\Sscr_K$, modeled on the induced representation $\mathbf{V}_I(\lambda)\colonequals \Ind_{\mathbf{B}}^{\mathbf{P}}(\lambda)$ (see section \ref{aut-VB-sec} for details). For any $\Ocal_{\mathbf{E}_\pfr}$-algebra $R$, we call elements of $H^0(\Sscr_K\otimes_{\Ocal_{\mathbf{E}_\pfr}} R, \Vcal_I(\lambda))$ automorphic forms of level $K$ and weight $\lambda$ with coefficients in $R$.

In the papers \cite{Goldring-Koskivirta-global-sections-compositio, Goldring-Koskivirta-divisibility}, Goldring and the author studied the set:
\begin{equation}\label{CKR-intro-eq}
    C_{K}(R)\colonequals \left\{ \lambda\in X^*(\mathbf{T}) \ \relmiddle| \  H^0(\Sscr_K\otimes_{\Ocal_{\mathbf{E}_\pfr}} R, \Vcal_I(\lambda))\neq 0\right\}.
\end{equation}
It is an additive submonoid (i.e. a "cone") inside $X^*(\mathbf{T})$. For a cone $C\subset X^*(\mathbf{T})$, we define the saturation of $C$ as the subset of $\lambda\in X^*(\mathbf{T})$ such that some positive multiple of $\lambda$ lies in $C$. We always write the saturation with a calligraphic letter $\Ccal$. The saturation of $C_K(R)$ is independent of $K$ (\cite[Corollary 1.5.3]{Koskivirta-automforms-GZip}), so we simply denote it by $\Ccal(R)$. When $R=\CC$ we proved in \cite{Goldring-Koskivirta-GS-cone} that $\Ccal(\CC)$ is contained in the Griffiths--Schmid cone, defined by
\begin{equation}
\Ccal_{\GS}=\left\{ \lambda\in X^{*}(\mathbf{T}) \ \relmiddle| \ 
\parbox{6cm}{
$\langle \lambda, \alpha^\vee \rangle \geq 0 \ \textrm{ for }\alpha\in I, \\
\langle \lambda, \alpha^\vee \rangle \leq 0 \ \textrm{ for }\alpha\in \Phi^+ \setminus \Phi^+_{\mathbf{L}}$}
\right\}.
\end{equation}
It is expected that $\Ccal(\CC)  = \Ccal_{\GS}$ (this equality seems to be well-known to experts).

In this paper, we are interested in the case when $R=\overline{\FF}_p$. In this case, there is no simple, general description of $\Ccal(\overline{\FF}_p)$ in terms of the root datum. We conjectured that the cone $\Ccal(\overline{\FF}_p)$ is entirely determined by the stack of $G$-zips $\GZip^\mu$, introduced by Moonen--Wedhorn (\cite{Moonen-Wedhorn-Discrete-Invariants}) and Pink--Wedhorn--Ziegler (\cite{Pink-Wedhorn-Ziegler-zip-data, Pink-Wedhorn-Ziegler-F-Zips-additional-structure}). Here $G$ denotes the $\FF_p$-reductive group $G\colonequals \mathbf{G}_{\ZZ_p}\otimes_{\ZZ_p}\FF_p$. The stack $\GZip^\mu$ is a smooth, finite stack over $\overline{\FF}_p$. By a result of Zhang (\cite{Zhang-EO-Hodge}), there is a natural smooth (surjective) map
\begin{equation}
    \zeta\colon S_K\to \GZip^\mu.
\end{equation}
The vector bundles $\Vcal_I(\lambda)$ also exist naturally on the stack $\GZip^\mu$. We define a cone $C_{\zip}\subset X^*(\mathbf{T})$ as the set of $\lambda\in  X^*(\mathbf{T})$ such that $H^0(\GZip^\mu,\Vcal_I(\lambda))\neq 0$, similarly to $C_K(R)$. Pullback via $\zeta$ yields an inclusion $H^0(\GZip^\mu,\Vcal_I(\lambda))\subset H^0(S_K,\Vcal_I(\lambda))$, hence we have the containment $C_{\zip}\subset C_K(\overline{\FF}_p)$. Passing to the saturation, we also have $\Ccal_{\zip}\subset \Ccal(\overline{\FF}_p)$. We conjectured (\cite[Conjecture C]{Goldring-Koskivirta-global-sections-compositio}):
\begin{conj1}
For any Hodge-type Shimura variety, one has $\Ccal(\overline{\FF}_p)  = \Ccal_{\zip}$.
\end{conj1}
Goldring and the author proved this conjecture in \loccit for Hilbert--Blumenthal Shimura varieties, Picard modular surfaces at split primes and Siegel threefolds. We treated the case of unitary Shimura varieties attached to $\GU(r,s)$ for $r+s\leq 4$ as well as $\GSp_6$ in \cite{Goldring-Koskivirta-divisibility} (with the exception of $\GU(2,2)$ at at inert prime). Since the cone $\Ccal_{\zip}$ is in general very difficult to determine, Conjecture 1 does not give an explicit expression for $\Ccal(\overline{\FF}_p)$. However, in the paper \cite{Goldring-Koskivirta-GS-cone}, we gave a very sharp and explicit approximation (from above) for $\Ccal_{\zip}$. Hence, Conjecture 1 predicts that this should also provide an approximation of $\Ccal(\overline{\FF}_p)$. Assume for simplicity thay $G$ is split over $\FF_{p^2}$ and that $p$ is split in $\mathbf{E}$. Let $W_{\mathbf{L}}=W(\mathbf{L},\mathbf{T})$ be the Weyl group of $\mathbf{L}$. Note that $W_{\mathbf{L}} \rtimes \Gal(\FF_{p^2}/\FF_p)$ acts naturally on the set $\Phi^+\setminus \Phi^+_{\mathbf{L}}$. Using the approximation of $\Ccal_{\zip}$ given in \cite{Goldring-Koskivirta-GS-cone}, Conjecture 1 implies the following, much more explicit conjecture.

\begin{conj2}
Let $S_K$ be the special fiber of a Hodge-type Shimura variety at a prime $p$ of good reduction which splits in the reflex field $\mathbf{E}$. Assume that the attached reductive $\FF_p$-group $G$ is split over $\FF_{p^2}$. Then, if $f\in H^0(S_K,\Vcal_I(\lambda))$ is a nonzero automorphic form of weight $\lambda\in X^*(\mathbf{T})$, we have 
\begin{equation}
    \sum_{\alpha \in \Ocal\setminus S} \langle \lambda,\alpha^\vee \rangle \ + \ \frac{1}{p} \sum_{\substack{\alpha\in S}} \langle  \lambda, \alpha^\vee \rangle \leq 0
\end{equation}
for all $W_{\mathbf{L}}\rtimes \Gal(\FF_{p^2}/\FF_p)$-orbits $\Ocal\subset \Phi^+\setminus \Phi^+_{\mathbf{L}}$ and all subsets $S\subset \Ocal$.
\end{conj2}
The above inequalities give a very good approximation (from above) of $\Ccal_{\zip}$ (see \cite[Figure 1]{Goldring-Koskivirta-divisibility} for a visual illustration of these cones in the case $G=\GSp_{6}$). The assumption that $p$ is split in $\mathbf{E}$ implies that the ordinary locus of $S_K$ is non-empty. We proved that Conjecture 2 holds for unitary Shimura varieties of signature $(n-1,1)$ at split primes in \cite{Goldring-Koskivirta-GS-cone}. We prove in this paper that Conjecture 2 holds for Siegel-type Shimura varieties. However, we cannot prove the stronger Conjecture 1 for these varieties (except for $n\leq 3$), and we cannot even determine $\Ccal_{\zip}$ for $n\geq 4$. In the case of $\GSp_{2n}$, characters are parametrized by an $n+1$-tuple $\lambda=(a_1,\dots,a_n,b)\in \ZZ^{n+1}$ satisfying $\sum_{i=1}^n a_i \equiv b \pmod{2}$ (where the natural multiplier of $\GSp_{2n}$ corresponds to $(0,\dots,0,2)$). Concretely, we show the following:

\begin{thm1}[Theorem \ref{main-thm}]\label{th-CF-Min-intro}
Conjecture 2 holds for good reduction Siegel-type Shimura varieties $\Ascr_{n,K}$. Specifically, if $f\in H^0(\Ascr_{n,K}\otimes_{\ZZ_p} \overline{\FF}_p, \Vcal_I(\lambda))$ is a nonzero automorphic form with coefficients in $\overline{\FF}_p$ of weight $\lambda=(a_1,\dots,a_n,b)$, we have: 
\begin{equation}\label{ineq-thm-intro} \tag{1}
    \sum_{i=1}^j a_i + \frac{1}{p} \sum_{i=j+1}^{n} a_i \leq 0 \quad \textrm{ for all }j=1,\dots, n.
\end{equation}
\end{thm1}

Equivalently, if for any $j\in \{1,\dots, n\}$ the inequality \eqref{ineq-thm-intro} is not satisfied, then the space $H^0(\Ascr_{n,K}\otimes_{\FF_p} \overline{\FF}_p, \Vcal_I(\lambda))$ is zero. Therefore, we may view the above result as a vanishing theorem for the cohomology of Shimura varieties in degree $0$. To prove Theorem 1, we use the flag space $\Flag(\Ascr_{n,K})$ of $\Ascr_{n,K}$, introduced by Ekedahl--van der Geer in \cite{Ekedahl-Geer-EO}. It is a moduli space which parametrizes pairs $(x,\Fcal_\bullet)$, where $x$ is a point of $\Ascr_{n,K}$ and $\Fcal_\bullet$ is a full symplectic flag in the de Rham cohomology of the abelian variety at $x$ which refines the Hodge filtration. There is a natural projection map
\begin{equation}
    \pi\colon \Flag(\Ascr_{n,K}) \to \Ascr_{n,K}, \qquad (x,\Fcal_\bullet)\mapsto x.
\end{equation}
For any character $\lambda \in X^*(\mathbf{T})$, there is a line bundle $\Vcal_{\flag}(\lambda)$ on $\Flag(\Ascr_{n,K})$ such that $\pi_*(\Vcal_{\flag}(\lambda))=\Vcal_I(\lambda)$. In particular, automorphic forms of weight $\lambda$ coincide with global sections of $\Vcal_{\flag}(\lambda)$ on the flag space. The special fiber $\Flag_K \colonequals \Flag(\Ascr_{n,K})\otimes_{\ZZ_p} \overline{\FF}_p$ admits a stratification $(\Flag_{K,w})_{w\in W}$ where $w$ varies in the Weil group $W$ of $G=\GSp_{2n}$. Write $\overline{\Flag}_{K,w}$ for the Zariski closure of $\Flag_{K,w}$. Set $w_{\max}=w_{0,I}w_0$, where $w_0$ and $w_{0,I}$ are the longest elements of $W$ and $W_{\mathbf{L}}$ respectively. The stratum $\Flag_{K,w_{\max}}$ plays a special role. It has the property that $\pi$ restricts to a finite etale map from $\Flag_{K,w_{\max}}$ onto the ordinary locus of $\Ascr_{n,K}\otimes_{\ZZ_{(p)}} \overline{\FF}_p$. The following theorem is the first step in the proof of Theorem 1.

\begin{thm2}[Theorem \ref{thm-Hasse-reg-wmax}]
Let $f\in H^0(\overline{\Flag}_{K,w_{\max}}, \Vcal_{\flag}(\lambda))$ be a nonzero section, where $\lambda=(a_1,\dots,a_n,b)$. Then we have $a_i \leq 0$ for all $i=1, \dots, n$.
\end{thm2}

The above result is proved by embedding Hilbert--Blumenthal Shimura varieties (attached to a totally real field $\mathbf{F}$ where $p$ splits) into $\Ascr_{n,K}$. This embedding lifts naturally to an embedding into the flag space of $\Ascr_{n,K}$, and on the special fiber the image is contained in the Zariski closure $\overline{\Flag}_{K,w_{\max}}$. Since Conjecture 1 holds for Hilbert--Blumenthal varieties by previous results, we deduce that if $a_i>0$ for some $1\leq i\leq n$, then any section $f\in H^0(\overline{\Flag}_{K,w_{\max}}, \Vcal_{\flag}(\lambda))$ must vanish on the image of this embedding. The final step to prove Theorem 2 is to use Hecke operators to show that $f$ vanishes everywhere. For this, define for $w\in W$ a set $\BB_{K,w}(\lambda)$ as the stable base locus of the line bundle $\Vcal_{\flag}(\lambda)$ restricted to $\overline{\Flag}_{K,w}$ (the stable base-locus is the intersection of the base loci of all positive powers of $\Vcal_{\flag}(\lambda)$). We show:
\begin{thm3}[Theorem \ref{B-hecke-stable}]
The set $\BB_{K,w}(\lambda)\subset \Flag_{K}$ is stable by prime-to-$p$ Hecke operators.
\end{thm3}
This makes it possible to apply a well-known Theorem of Chai on the density of prime-to-$p$ Hecke orbits of ordinary points to terminate the proof of Theorem 2. Lastly, to prove Theorem 1, we use the flag stratum $\Flag_{K,w_{\max}}$ as a starting point and propagate the inequalities afforded by Theorem 2 using an appropriately chosen increasing sequence of elements in $W$, until we reach the maximal element $w_0$. This last step is similar to the method used in \cite{Goldring-Koskivirta-GS-cone} in the case of unitary groups.

\vspace{0.8cm}

\noindent
{\bf Acknowledgements.}
This work was supported by JSPS KAKENHI Grant Number 21K13765. We would like to thank Wushi Goldring for useful discussions on the topic of this paper. We are also very grateful to Christophe Cornut for many useful comments.

\section{Shimura varieties}

\subsection{Siegel-type Shimura varieties} \label{Siegel-type-sec}

Let $n\geq 1$ be an integer and let $\Lambda=\ZZ^{2n}$, endowed with the symplectic pairing $\Psi\colon \Lambda\times \Lambda \to \ZZ$ defined by the $2n\times 2n$ symplectic matrix
\begin{equation}
   \left( \begin{matrix}
        & J_0 \\ -J_0 &
    \end{matrix} \right), \quad  \textrm{where} \quad \quad J_0=
    \left( \begin{matrix}
        & &1 \\  & \iddots & \\ 1& &
    \end{matrix} \right) \in \GL_n(\ZZ).
\end{equation}
Let $\GSp_{2n}$ be the reductive group over $\ZZ$ such that for any $\ZZ$-algebra $R$, we have
\begin{equation}
    \GSp_{2n}(R)=\left\{ g\in \GL_R(\Lambda \otimes_\ZZ R) \ \relmiddle| \ \exists c(g)\in R^\times, \ \Psi_R(gx,gy) = c(g) \Psi_R(x,y) \right\}.
\end{equation}
The map $\GSp_{2n}\to \GG_{\mathrm{m}}$, $g\mapsto c(g)$ is the multiplier character. The subgroup defined by the condition $c(g)=1$ is denoted by $\Sp_{2n}$. Let $\Hcal_n^{\pm}$ be the set of symmetric complex matrices of size $n\times n$ whose imaginary part is positive definite or negative definite. The pair $(\GSp_{2n,\QQ},\Hcal_{n}^{\pm})$ is called a Shimura datum of Siegel-type.

We fix a prime number $p$ and set $K_p\colonequals \GSp_{2n}(\ZZ_p)$. It is a hyperspecial subgroup of $\GSp_{2n}(\QQ_p)$. For any open compact subgroup $K^p\subset \GSp_{2n}(\AA^p_f)$, write $K=K_p K^p \subset \GSp_{2n}(\AA_f)$ and let $\Ascr_{n,K}$ be stack over $\ZZ_{(p)}$ such that for any $\ZZ_{(p)}$-scheme $S$, the $S$-valued points of $\Ascr_{n,K}$ parametrize the triples $(A,\lambda, \overline{\eta})$ satisfying the following conditions:
\begin{bulletlist}
    \item $A$ is an abelian scheme of relative dimension $n$ over $S$.
    \item $\lambda\colon A\to A^\vee$ is a $\ZZ_{(p)}$-multiple of a polarization whose degree is prime to $p$.
    \item $\eta\colon \Lambda \otimes_{\ZZ} \AA^p_f \to H^1(A,\AA^p_f)$ is an isomorphism of sheaves of $\AA^p_f$-modules on $S$. We impose that $\eta$ is compatible with the symplectic pairings induced by $\Psi$ and $\lambda$. We write $\overline{\eta}=\eta K^p$ for the $K^p$-orbit of $\eta$.
\end{bulletlist}
For $K^p$ small enough, $\Ascr_{n,K}$ is a smooth, quasi-projective $\ZZ_{(p)}$-scheme of relative dimension $\frac{n(n+1)}{2}$. We will always make this assumption on $K^p$. We are interested in the geometry of the special fiber $\Ascr_{n,K}\otimes_{\ZZ_{(p)}} \FF_p$. We fix an algebraic closure $k=\overline{\FF}_p$ and define:
\begin{equation}
    \overline{\Ascr}_{n,K} \colonequals \Ascr_{n,K}\otimes_{\ZZ_{(p)}} k.
\end{equation}


\subsection{Hodge-type Shimura varieties}\label{sec-HT-SV}

More generally, let $\gx$ be a Shimura datum of Hodge-type \cite[2.1.1]{Deligne-Shimura-varieties}. Recall that this means that there is an embedding of Shimura data $\gx\to (\GSp_{2n,\QQ},\Hcal_n^{\pm})$ for some $n\geq 1$. Here $\mathbf{G}$ is a connected, reductive group over $\QQ$. The symmetric domain $\mathbf{X}$ gives rise to a well-defined $\mathbf{G}(\overline{\QQ})$-conjugacy class of cocharacters $[\mu]$ of $\mathbf{G}_{\overline{\QQ}}$. Write $\mathbf{E}=\mathbf{E}(\mathbf{G},\mathbf{X})$ for the reflex field of $\gx$ (i.e. the field of definition of $[\mu]$). For any open compact subgroup $K \subset \mathbf{G}(\AA_f)$, let $\Sh_{K}(\mathbf{G},\mathbf{X})$ be Deligne's canonical model (\cite{Deligne-Shimura-varieties}) at level $K$ defined over $\mathbf{E}$. When $K$ is small enough, $\Sh_{K}(\mathbf{G},\mathbf{X})$ is a smooth, quasi-projective scheme over $\mathbf{E}$, and its $\CC$-valued points are given by
\begin{equation}
    \Sh_{K}(\mathbf{G},\mathbf{X})(\CC) = \mathbf{G}(\QQ) \backslash \left(\mathbf{G}(\AA_f)/K \times \mathbf{X} \right).
\end{equation}
Let $p$ be a prime of good reduction. By this, we mean that $K$ can be written in the form $K=K_p K^p$ where $K_p\subset \mathbf{G}(\QQ_p)$ is hyperspecial and $K^p\subset \mathbf{G}(\AA_f^p)$ is open compact. We will only consider open compact subgroups of this form. In particular, the group $\mathbf{G}_{\QQ_p}$ is unramified and there exists a connected reductive $\ZZ_p$-model $\mathbf{G}_{\ZZ_p}$ such that $K_p=\mathbf{G}_{\ZZ_p}(\ZZ_p)$. For any place $\pfr$ above $p$ in $\mathbf{E}$, Kisin (\cite{Kisin-Hodge-Type-Shimura}) and Vasiu (\cite{Vasiu-Preabelian-integral-canonical-models}) constructed a smooth canonical model $\Sscr_K$ of $\Sh_{K}(\mathbf{G},\mathbf{X})$ over $\Ocal_{\mathbf{E}_\pfr}$. The embedding $\gx\to (\GSp_{2n,\QQ},\Hcal_n^{\pm})$ induces an embedding $\Sh_K(\mathbf{G},\mathbf{X})\to \Ascr_{n,K_0}\otimes_{\ZZ_{(p)}}\mathbf{E}_{\pfr}$ for a compatible choice of open compact subgroups $K^p_0\subset \GSp_{2n}(\AA^p_f)$, $K^p\subset \mathbf{G}(\AA_f^p)$. By \cite[Theorem 2.4.8]{Kisin-Hodge-Type-Shimura}, $\Sscr_K$ is the normalization of the scheme-theoretical image of $\Sh_K(\mathbf{G},\mathbf{X})$ inside $\Ascr_{n,K_0}$. Put
\begin{equation}
    S_{K}\colonequals \Sscr_K\otimes_{\Ocal_{\mathbf{E}_\pfr}} k.
\end{equation}
Let $K'^p\subset K^p$ be two compact open subgroups of $\mathbf{G}(\AA_f^p)$. There is a natural projection morphism
\begin{equation}\label{change-level-pi}
    \pi_{K',K}\colon \Sscr_{K'}\to \Sscr_K
\end{equation}
which is finite etale. We call $\pi_{K',K}$ the change-of-level map.

\subsection{Hecke operators}

We review the definition of Hecke operators on $\Sscr_{K}$. For any $g\in \mathbf{G}(\AA_f^p)$, consider the compact open subgroup $K'=K\cap gKg^{-1}$. There are two natural maps $\Sscr_{K\cap g K g^{-1}}\to \Sscr_{K}$. The first one is the natural change-of-level map $\pi_{K',K}$ defined in \eqref{change-level-pi}. The second one is induced on $\CC$-points by the map
\begin{equation}
    \mathbf{G}(\AA_f)/K' \to \mathbf{G}(\AA_f)/K, \quad x K' \mapsto gx K.
\end{equation}
We denote the induced map simply by $g\colon \Sscr_{K\cap g K g^{-1}}\to \Sscr_K$. This construction yields a correspondance with finite etale maps

\begin{equation}\label{hecke-diag}
    \xymatrix{
& \Sscr_{K\cap g K g^{-1}} \ar[dl]_{\pi}  \ar[dr]^{g} & \\
\Sscr_{K} &  & \Sscr_{K}
}
\end{equation}
The Hecke algebra $\TT_{K^p}(\mathbf{G})$ (with coefficients in $\ZZ$) attached to $K^p$ is the ring of $K^p$-bi-invariant, locally constant functions $\mathbf{G}(\AA_f^p)\to \ZZ$ with compact support. Any such function can be written as a sum of characteristic functions of double cosets $K^p g K^p$ ($g\in \mathbf{G}(\AA_f^p)$). Multiplication is defined by convolution with respect to the left-Haar measure $\nu$ normalized by $\nu(K^p)=1$. The algebra $\TT_{K^p}(\mathbf{G})$ acts on several natural objects related to $\Sscr_{K}$ via the correspondance \eqref{hecke-diag}. We will only consider the action on 0-cycles on the special fiber of $\Sscr_{K}$. Concretely, if $x\in S_{K}(k)$ and $g\in \mathbf{G}(\AA_f^p)$, we define
\begin{equation}\label{TKp}
    T_{K^p,g}(x) = g_{*} \pi^*(x).
\end{equation}
We view $T_{K^p,g}(x)$ as a formal sum of points of $S_{K}(k)$. We denote by $\Hcal^p(x)$ the set of all points of $S_{K}(k)$ appearing in the support of a 0-cycle of the form $T_{K^p,g}(x)$ for $g\in \mathbf{G}(\AA_f^p)$. We call $\Hcal^p(x)$ the prime-to-$p$ Hecke orbit of $x$. Denote by $S_K^{\ord}$ the ordinary locus of $S_K$, i.e. the set of points $x\in S_K$ whose image in $\Ascr_{n,K_0}\otimes_{\ZZ_{(p)}} k$ corresponds to an ordinary abelian variety. $S^{\ord}_{K}$ is nonempty if and only if $\mathbf{E}_{\pfr}=\QQ_p$. By \cite[Theorem I]{van-hoften-ordinary-hecke}, one has the following theorem:

\begin{theorem}\label{thm-Hx-dense}
Assume that $\mathbf{E}_{\pfr}=\QQ_p$. For any point $x\in S_K^{\ord}$, the set $\Hcal^p(x)$ is Zariski dense in $S_{K}$.
\end{theorem}
In the case of Siegel-type Shimura varieties, Theorem \ref{thm-Hx-dense} was first proved by Chai in \cite{Chai-ordinary-isogeny-invent-math}.

\subsection{Hilbert--Blumenthal Shimura varieties} \label{hb-var-sec}

In this section, we recall the definition of Hilbert--Blumenthal Shimura varieties (also known as Hilbert modular varieties). Let $\mathbf{F}/\QQ$ be a totally real extension of degree $n$. We denote by $\mathbf{H}$ the reductive $\QQ$-group whose $R$-points (for any $\QQ$-algebra $R$) are defined by
\begin{equation}
    \mathbf{H}(R)=\left\{ g\in \GL_2(\mathbf{F}\otimes_\QQ R) \ \relmiddle| \ \det(g)\in R^\times \right\}.
\end{equation}
It is a subgroup of the Weil restriction $\Res_{\mathbf{F}/\QQ}(\GL_{2,\mathbf{F}})$. We assume that $p$ is unramified in $\mathbf{F}$. In this case, the group $\mathbf{H}$ is unramified at $p$ and the lattice $\Ocal_{\mathbf{F}}\otimes_{\ZZ}\ZZ_p \subset \mathbf{F}\otimes_{\QQ}\QQ_p$ yields a reductive $\ZZ_p$-model $\mathbf{H}_{\ZZ_p}$ of $\mathbf{H}_{\QQ_p}$. We set $K'_p\colonequals \mathbf{H}_{\ZZ_p}(\ZZ_p)$, it is a hyperspecial subgroup of $\mathbf{H}(\QQ_p)$. For any open compact subgroup $K'^p \subset \mathbf{H}(\AA_f^p)$, we write $K'\colonequals K'_p K'^{p}$. The Hilbert--Blumenthal Shimura variety $\Hscr_{\mathbf{F},K'}$ of level $K'$ is the $\ZZ_{(p)}$-stack whose $S$-valued points (for any $\ZZ_{(p)}$-scheme $S$) parametrize quadruples $(A,\lambda, \iota, \overline{\eta})$ satisfying the following conditions:
\begin{bulletlist}
    \item $A$ is an abelian scheme of relative dimension $n$ over $S$.
    \item $\lambda\colon A\to A^\vee$ is a $\ZZ_{(p)}$-multiple of a polarization of degree prime to $p$.
    \item $\iota\colon \Ocal_\mathbf{F} \otimes_\ZZ \ZZ_{(p)} \to \End(A)\otimes_\ZZ \ZZ_{(p)}$ is a ring homomorphism.
    \item $\eta\colon (\mathbf{F} \otimes_{\QQ} \AA^p_f)^2 \to H^1(A,\AA^p_f)$ is an $\mathbf{F}$-linear isomorphism of sheaves of $\AA^p_f$-modules on $S$, and $\overline{\eta}=\eta K'^p$ is the $K'^p$-orbit of $\eta$.
\end{bulletlist}
The homomorphism $\iota$ yields an action of $\Ocal_{\mathbf{F}}$ on the dual abelian variety $A^\vee$. We impose that the polarization $\lambda$ be $\Ocal_\mathbf{F}$-linear for this action.

Next, we construct a morphism $\Hscr_{\mathbf{F},K'}\to \Ascr_{n,K}$ for an appropriate choice of level structures $K'^p\subset \mathbf{H}(\AA_f^p)$ and $K^p\subset \GSp(\AA_f^p)$. Write $\mathbf{G}\colonequals \GSp_{2n,\QQ}$ and $\mathbf{G}_{\ZZ_p}=\GSp_{2n,\ZZ_p}$. First, we construct an embedding $\mathbf{H} \to \mathbf{G}$ of reductive $\QQ$-groups. Consider the symplectic form
\begin{equation}
    \Psi_{0}\colon \mathbf{F}^2\times \mathbf{F}^2\to \QQ, \quad ((x_1,y_1),(x_2,y_2)) \ \mapsto \ \Tr_{\mathbf{F}/\QQ}(x_1y_2-x_2y_1).
\end{equation}
For any matrix $A\in \GL_2(\mathbf{F})$ such that $\det(A)\in \QQ^\times$ and $x,y\in \mathbf{F}^2$, one has
\begin{equation}
    \Psi_0(Ax,Ay)=\det(A)\Psi_0(x,y).
\end{equation}
Hence, we have an inclusion $\mathbf{H}\subset \GSp(\Psi_0)$. Fix an isomorphism $\gamma\colon (\mathbf{F}^2,\Psi_0)\to (\Lambda\otimes_{\ZZ}\QQ, \Psi)$ of symplectic spaces over $\QQ$. We obtain an embedding of reductive $\QQ$-groups
\begin{equation}\label{u-embedQ}
u\colon \mathbf{H} \to \mathbf{G}, \quad f\mapsto \gamma \circ f \circ \gamma^{-1}.
\end{equation}
Since $p$ is unramified in $\mathbf{F}$, we may further assume that $\gamma$ induces an isomorphism $\Ocal_{\mathbf{F}}\otimes_{\ZZ}\ZZ_p\to \Lambda\otimes_{\ZZ} \ZZ_p$. Then, $u_{\QQ_p}$ extends to an embedding $u_{\ZZ_p}\colon \mathbf{H}_{\ZZ_p} \to \mathbf{G}_{\ZZ_p}$ of reductive $\ZZ_{p}$-groups. For any compact open subgroups $K'^p\subset \mathbf{H}(\AA_f^p)$ and $K^p\subset \mathbf{G}(\AA_f^p)$ such that $u(K'^p)\subset K^p$, there is a natural morphism of $\ZZ_{(p)}$-schemes
\begin{equation}\label{embed-HB-S}
\widetilde{u}_{K',K} \colon \Hscr_{\mathbf{F},K'} \to \Ascr_{n,K}
\end{equation}
defined as follows. Let $x=(A,\lambda,\iota,\overline{\eta})$ be a point of $\Hscr_{\mathbf{F},K'}$, where $\overline{\eta}=\eta K'^p$ and $\eta\colon (\mathbf{F} \otimes_{\QQ} \AA^p_f)^2 \to H^1(A,\AA^p_f)$ is an $\mathbf{F}$-linear isomorphism. Then $\widetilde{u}_{K',K}$ sends $x$ to the point $(A,\lambda, (\eta\circ \gamma^{-1})K^p)$ of $\Ascr_{n,K}$ (the isomorphism $\eta\circ \gamma^{-1}\colon \Lambda\otimes_{\ZZ} \AA^p_f \to H^1(A,\AA^p_f)$ is compatible with the symplectic forms, and thus gives rise to a level structure). It is possible to choose $K', K$ so that $\widetilde{u}_{K',K}$ is a closed embedding, but we will not need this.

To conclude this section, we consider the case when $p$ splits in $\mathbf{F}$. If we let $\pfr_1,\dots,\pfr_n$ be the prime ideals of $\mathbf{F}$ dividing $p$, we get an identification 
\begin{equation}
\Ocal_{\mathbf{F}}\otimes_{\ZZ} \FF_p \simeq \prod_{i=1}^n \Ocal_{\mathbf{F}}/\pfr_i \simeq \FF_p^n.   
\end{equation}
Write $H\colonequals \mathbf{H}_{\ZZ_p}\otimes_{\ZZ_p}\FF_p$ and $G\colonequals \mathbf{G}_{\ZZ_p}\otimes_{\ZZ_p}\FF_p$ for the special fibers. We obtain an isomorphism
\begin{equation}
    H\simeq \{(A_1,\dots,A_n)\in \GL_{2,\FF_p} \ | \ \det(A_1)=\dots=\det(A_n)\}.
\end{equation}
Furthermore, in an appropriate basis the embedding $u\colon H\to G$ identifies with the following map.
\begin{equation}\label{emb-split}
    u \colon \left(
\left(
\begin{matrix}
    a_1&b_1\\ c_1&d_1
\end{matrix}
\right),
\cdots , 
\left(
\begin{matrix}
    a_n&b_n\\ c_n&d_n
\end{matrix}
\right)
    \right) \mapsto 
    \left(
\begin{matrix}
a_1 & & & & & & & b_1 \\
&a_2 & & & & & b_2&  \\
& & \ddots& & &\iddots & & \\
& & & a_n &b_n & & & \\
& & & c_n& d_n& & & \\
& &\iddots & & &\ddots & & \\
& c_2& & & & &d_2 & \\
c_1 & & & & & & &d_1 \\
\end{matrix}
    \right).
\end{equation}

\section{\texorpdfstring{The stack of $G$-zips}{}}\label{subsec-GZip}

The stack of $G$-zips was introduced by Moonen--Wedhorn (\cite{Moonen-Wedhorn-Discrete-Invariants}) and Pink--Wedhorn--Ziegler (\cite{Pink-Wedhorn-Ziegler-zip-data, Pink-Wedhorn-Ziegler-F-Zips-additional-structure}). We start be recalling its definition.

\subsection{Definition}\label{sec-GZip-general}
Fix an algebraic closure $k=\overline{\FF}_p$ of $\FF_p$. For a $k$-scheme $X$, we denote by $X^{(p)}$ its $p$-th power Frobenius twist and by $\varphi\colon X\to X^{(p)}$ its relative Frobenius. Let $G$ be a connected, reductive group over $\FF_p$ endowed with a cocharacter $\mu\colon \GG_{\mathrm{m},k}\to G_k$. We call the pair $(G,\mu)$ a cocharacter datum. The cocharacter $\mu$ induces a decomposition $\Lie(G)=\bigoplus_{n\in \ZZ}\Lie(G)_n$, where $\Lie(G)_n$ is the subspace where $\GG_{\mathrm{m},k}$ acts by the character $x\mapsto x^n$ via $\mu$. We obtain a pair of opposite parabolic subgroups $P_{\pm}\subset G_k$, defined by the conditions
\begin{equation}\label{Pminplus}
    \Lie(P_-)=\bigoplus_{n\leq 0} \Lie(G)_n \quad \textrm{and} \quad \Lie(P_+)=\bigoplus_{n\geq 0} \Lie(G)_n.
\end{equation}
We set $P\colonequals P_-$ and $Q\colonequals P_+^{(q)}$. Let $L\colonequals \cent(\mu)$ be the centralizer of $\mu$, it is a Levi subgroup of $P$. Put $M\colonequals L^{(q)}$, which is a Levi subgroup of $Q$. We have a Frobenius map $\varphi\colon L\to M$. We call the tuple \[\Zcal_\mu \colonequals (G,P,Q,L,M)\]
the zip datum attached to $(G,\mu)$ (this terminology slightly differs from \cite[Definition 3.6]{Pink-Wedhorn-Ziegler-F-Zips-additional-structure}). If $L$ is defined over $\FF_p$, we have $M=L^{(p)}=L$. This will be the case for all zip data considered in this paper. Let $\theta_L^P\colon P\to L$ be the projection onto the Levi subgroup $L$ modulo the unipotent radical $R_{\mathrm{u}}(P)$ and define $\theta_M^Q\colon Q\to M$ similarly. The zip group is defined by
\begin{equation}\label{eq-Edef}
    E\colonequals\{ (x,y)\in P\times Q \mid \varphi(\theta^P_L(x))=\theta_M^Q(y) \}.
\end{equation}
Let $E$ act on $G_k$ by the rule $(x,y)\cdot g\colonequals xgy^{-1}$. The stack of $G$-zips of type $\mu$ is the quotient stack
\begin{equation}\label{GZip-def-eq}
    \GZip^\mu \colonequals \left[E\backslash G_k\right].
\end{equation}
It can also be defined as a moduli stack of certain torsors (\cite[Definition 1.4]{Pink-Wedhorn-Ziegler-F-Zips-additional-structure}).  The association $(G,\mu) \mapsto \GZip^\mu$ is functorial in the following sense. Let $(H,\mu_H)$ and $(G,\mu_G)$ be two cocharacter data and let $f\colon H\to G$ be a homomorphism defined over $\FF_p$ satisfying $\mu_G=f_k\circ \mu_H$. Then by \cite[\S2.1, \S2.2]{Goldring-Koskivirta-zip-flags}, $f$ induces a natural morphism of stacks $f_{\zip} \colon \HZip^{\mu_H}\to \GZip^{\mu_G}$ which makes the diagram below commute (the vertical maps are the natural projections).
\begin{equation}\label{GZip-functor}
\xymatrix@M=7pt{
H \ar[r]^{f} \ar[d] & G \ar[d] \\
\HZip^{\mu_H}\ar[r]_{f_{\zip}} & \GZip^{\mu_G}.
}
\end{equation}

\subsection{Parametrization of strata}\label{sec-param}

Let $(G,\mu)$ be a cocharacter datum, with attached zip datum $\Zcal_\mu=(G,P,Q,L,M)$. The stack $\GZip^\mu$ is a smooth stack over $k$ whose underlying topological space is finite. We review below the parametrization of the points of $\GZip^\mu$, following \cite{Pink-Wedhorn-Ziegler-zip-data}. We first fix some group-theoretical data. 
\begin{bulletlist}
\item For simplicity, we assume that there is a Borel pair $(B,T)$ defined over $\FF_p$ satisfying $B\subset P$, $T\subset L$ and such that $\mu$ factors through $T$ (this condition can always be achieved after possibly changing $\mu$ to a conjugate cocharacter). In particular, we have an action of $\Gal(k/\FF_p)$ on $X^*(T)$ and $X_*(T)$. Let $B^+$ denote the opposite Borel to $B$ with respect to $T$ (i.e. the only Borel subgroup such that $B\cap B^+=T$). We write $\sigma\in \Gal(k/\FF_p)$ for the $p$-power Frobenius automorphism $x\mapsto x^p$.
\item Let $W=W(G_k,T)$ be the Weyl group of $G_k$. The group $\Gal(k/\FF_p)$ acts on $W$ and the actions of $\Gal(k/\FF_p)$ and $W$ on $X^*(T)$ and $X_*(T)$ are compatible in a natural sense.
\item Write $\Phi$ for the set of $T$-roots, $\Phi^+$ for the positive roots with respect to $B$ (in our convention, a root $\alpha$ is positive if the corresponding $\alpha$-root group $U_\alpha$ is contained in the opposite Borel $B^+$). Let $\Delta$ denote the set of simple roots.
\item For $\alpha \in \Phi$, let $s_\alpha \in W$ be the corresponding reflection. The system $(W,\{s_\alpha \mid \alpha \in \Delta\})$ is a Coxeter system.
\item Write $\ell  \colon W\to \NN$ for the length function, and $\leq$ for the Bruhat order on $W$. Let $w_0$ denote the longest element of $W$. 
\item Write $I\subset \Delta$ for the set of simple roots contained in $L$.
\item For a subset $K\subset \Delta$, let $W_K$ denote the subgroup of $W$ generated by $\{s_\alpha \mid \alpha \in K\}$. Write $w_{0,K}$ for the longest element in $W_K$.
\item Let ${}^KW$ denote the subset of elements $w\in W$ which have minimal length in the coset $W_K w$. Then ${}^K W$ is a set of representatives of $W_K\backslash W$. The longest element in the set ${}^K W$ is $w_{0,K} w_0$.
\item We set throughout
\begin{equation}\label{z-def}
    z=\sigma(w_{0,I})w_0.
\end{equation}
The triple $(B,T,z)$ is a $W$-frame, in the terminology of \cite[Definition 2.3.1]{Goldring-Koskivirta-zip-flags} (we simply call such a triple a frame).
\item For $w\in W$, fix a representative $\dot{w}\in N_G(T)$, such that $(w_1w_2)^\cdot = \dot{w}_1\dot{w}_2$ whenever $\ell(w_1 w_2)=\ell(w_1)+\ell(w_2)$ (this is possible by choosing a Chevalley system, \cite[ XXIII, \S6]{SGA3}). If no confusion occurs, we simply write $w$ instead of $\dot{w}$.
\item For $w,w'\in {}^I W$, write $w'\preccurlyeq w$ if there exists $w_1\in W_I$ such that $w'\leq w_1 w \sigma(w_1)^{-1}$. This defines a partial order on ${}^I W$ (\cite[Corollary 6.3]{Pink-Wedhorn-Ziegler-zip-data}).
\item For $w\in {}^I W$, define $G_w$ as the $E$-orbit of $\dot{w}\dot{z}^{-1}$.
\end{bulletlist}

\bigskip

We now explain the parametrization of $E$-orbits in $G_k$.

\begin{theorem}[{\cite[Theorem 7.5]{Pink-Wedhorn-Ziegler-zip-data}}] \label{thm-E-orb-param} \ 
\begin{assertionlist}
\item Each $E$-orbit is smooth and locally closed in $G_k$.
\item The map $w\mapsto G_w$ is a bijection from ${}^I W$ onto the set of $E$-orbits in $G_k$.
\item For $w\in {}^I W$, one has $\dim(G_w)= \ell(w)+\dim(P)$.
\item The Zariski closure of $G_w$ is 
\begin{equation}\label{equ-closure-rel}
\overline{G}_w=\bigsqcup_{w'\in {}^IW,\  w'\preccurlyeq w} G_{w'}.
\end{equation}
\end{assertionlist}
\end{theorem}

For each $w\in {}^I W$, we define the quotient stack
\begin{equation}
    \Xcal_w \colonequals [E\backslash G_w].
\end{equation}
It is a locally closed substack of $\GZip^\mu=[E\backslash G_k]$. We call the $\Xcal_w$ the zip strata of $\GZip^\mu$. We obtain a stratification $\GZip^\mu = \bigsqcup_{w\in {}^I W} \Xcal_w$ and the closure relations between strata are given by \eqref{equ-closure-rel}. In particular, there is a unique open $E$-orbit
$U_\mu\subset G$ corresponding to the longest element 
\begin{equation}
    w_{\max}\colonequals w_{0,I}w_0 
\end{equation}
in ${}^I W$. We always have $1\in U_\mu$. The open substack $\Ucal_\mu\colonequals \Xcal_{w_{\max}}$ is called the \emph{$\mu$-ordinary locus} of $\GZip^\mu$, following usual terminology pertaining to Shimura varieties.

\subsection{\texorpdfstring{$G$-zips and Shimura varieties of Hodge-type}{}}\label{sec-GZip-HT}

We retain the notations of section \ref{sec-HT-SV}. Let $S_K=\Sscr_K\otimes_{\Ocal_{\mathbf{E}_\pfr}} k$ be the special fiber of a Hodge-type Shimura variety at a place of good reduction. Define $G\colonequals \mathbf{G}_{\ZZ_p}\otimes_{\ZZ_p} \FF_p$. We can find a cocharacter $\mu$ in the conjugacy class $[\mu]$ such that $\mu$ extends to a cocharacter of $\mathbf{G}_{\ZZ_p}\otimes_{\ZZ_p} \Ocal_{\mathbf{E}_{\pfr}}$ (\cite[Corollary 3.3.11]{Kim-Rapoport-Zink-uniformization}). We denote again by $\mu$ the cocharacter of $G_k$ obtained by reducing modulo $p$. By \S\ref{sec-GZip-general}, there is a stack of $G$-zips of type $\mu$ attached to the cocharacter datum $(G,\mu)$. Zhang (\cite[4.1]{Zhang-EO-Hodge}) constructed a smooth morphism
\begin{equation}\label{zeta-Shimura}
\zeta_K \colon S_{K}\to \GZip^\mu,
\end{equation}
whose fibers are called the Ekedahl--Oort strata of $S_K$. This map is also surjective by \cite[Corollary 3.5.3(1)]{Shen-Yu-Zhang-EKOR}. The map $\zeta_K$ commutes with change-of-level maps, in the sense that for any compact open subgroups $K'^p\subset K^p \subset \mathbf{G}(\AA_f^p)$, there is a commutative diagram
\begin{equation}
\xymatrix@M=5pt{
     S_{K'} \ar[rd]^{\zeta_{K'}} \ar[dd]_{\pi_{K',K}} & \\
     & \GZip^\mu \\
     S_{K} \ar[ru]_{\zeta_K}
     }
\end{equation}
For this reason, we often omit the subscript $K$ and denote these maps by $\zeta$. For $w\in {}^I W$, write
\begin{equation}
    S_{K,w}\colonequals \zeta^{-1}(\Xcal_{w})
\end{equation}
for the corresponding Ekedahl--Oort stratum. The Ekedahl--Oort stratum $S_{K,w_{\max}}$ corresponding to the maximal element $w_{\max}=w_{0,I}w_0\in {}^I W$ is called the $\mu$-ordinary locus. When $\mathbf{E}_\pfr=\QQ_p$, this is simply the classical ordinary locus of $S_K$, and is often denoted by $S_K^{\ord}$.

\subsubsection*{The Siegel case}\label{siegel-zip-sec}
We consider the case $G=\GSp_{2n,\FF_p}$ (as defined in section \ref{Siegel-type-sec}). We endow $G$ with the cocharacter
\begin{equation}
    \mu_G\colon \GG_{\mathrm{m},k} \to G_k, \quad t\mapsto \left(
\begin{matrix}
    t I_n & \\ & I_n
\end{matrix}
    \right).
\end{equation}
Let $(u_1,\dots, u_{2n})$ denote the canonical basis of $V = k^{2n}$, and set \[V_P\colonequals \Span_k(u_{n+1},\dots, u_{2n}).\]
Let $P$ be the parabolic subgroup which stabilizes the filtration $0\subset V_P \subset V$. Similarly, let $Q$ be the parabolic subgroup opposite to $P$ stabilizing the subspace $V_Q\colonequals \Span_k(u_1,\dots, u_n)$. The intersection $L\colonequals P\cap Q$ is a common Levi subgroup to $P$ and $Q$. The zip datum attached to $(G,\mu)$ is $\Zcal_\mu = (G,P,Q,L,L)$. Let $E\subset P\times Q$ be the group defined in \eqref{eq-Edef} and let $\GZip^{\mu_G}\colonequals [E\backslash G_k]$ be the attached stack of $G$-zips. We fix $B\subset G$ to be the Borel subgroup of lower-triangular matrices in $G$, and we let $T\subset B$ be the maximal torus consisting of diagonal matrices in $G$. The Weyl group $W=W(G,T)$ is isomorphic to $(\ZZ/2\ZZ)^n \rtimes \Sfr_n$ and can be identified with the group of permutations $w\in \Sfr_{2n}$ satisfying 
\begin{equation}\label{SpW-eq}
    w(i)+w(2n+1-i)=2n+1.
\end{equation}
The Weyl group $W_L=W(L,T)$ is isomorphic to $\Sfr_n$ and identifies with the subgroup of permutations in $\Sfr_{2n}$ satisfying \eqref{SpW-eq} and stabilizing the subset $\{1,\dots, n\}$. As explained in \S\ref{sec-GZip-HT}, there is a smooth surjective morphism
\begin{equation}
    \zeta_G\colon \overline{\Ascr}_{n,K}\to \GZip^{\mu_G}
\end{equation}
whose fibers are the Ekedahl--Oort strata of $\overline{\Ascr}_{n,K}$.

\subsubsection*{The Hilbert--Blumenthal case}
Next, we consider the Hilbert--Blumenthal case. We simply write $H$ for the special fiber of the reductive group $\textbf{H}_{\ZZ_p}$ defined in section \ref{hb-var-sec}. The embedding $u_{\ZZ_p}\colon \mathbf{H}_{\ZZ_p}\to \mathbf{G}_{\ZZ_p}$ induces an embedding 
\begin{equation}
    u\colon H\to G
\end{equation}
(to simplify the notation, we continue to denote it by $u$). If we let $\pfr_1,\dots,\pfr_r$ be the prime ideals of $\Ocal_{\mathbf{F}}$ above $p$ and write $\kappa_i\colonequals \Ocal_{\mathbf{F}}/\pfr_i$, we may view $H$ as a subgroup
\begin{equation}
    H\subset \prod_{i=1}^r \Res_{\kappa_i/\FF_p} \GL_{2,\kappa_i}.
\end{equation}
The group $H_k$ identifies with
\begin{equation} \label{Hk-ident}
    \left\{(A_1,\dots, A_n)\in \GL_{2,k}^n \ \relmiddle| \ \det(A_1)=\dots=\det(A_n) \right\}.
\end{equation}
We consider the stack of $H$-zips attached to the cocharacter
\begin{equation}
    \mu_H\colon \GG_{\mathrm{m},k}\to H_k, \quad t\mapsto \left(
\left( \begin{matrix}
    t & \\ & 1
\end{matrix}
    \right), \dots , \left( \begin{matrix}
    t & \\ & 1
\end{matrix}
    \right)
    \right).
\end{equation}
Write $B_H$ (resp. $B_H^{+}$) for the subgroup consisting of tuples $(A_1,\dots,A_n)\in H_k$ of lower-triangular (resp. upper-triangular) matrices. It is clear that $B_H$, $B_H^+$ are Borel subgroups of $H$ defined over $\FF_p$. Write $T_H=B_H\cap B_H^+$ for the maximal torus consisting of tuples of diagonal matrices in $H_k$. The zip datum $\Zcal_{\mu_H}$ attached to $\mu_H$ is $\Zcal_{\mu_H}=(H,B_H,B^+_H, T_H,T_H)$. Let $\HZip^{\mu_H}$ be the stack of $H$-zips attached to $(H,\mu_H)$. Note that the embedding $u \colon H\to G$ is compatible with the cocharacters $\mu_H$ and $\mu_G$. Therefore, $u$ induces by functoriality (see diagram \eqref{GZip-functor}) a morphism of stacks
\begin{equation}
    u_{\zip}\colon \HZip^{\mu_H} \to \GZip^{\mu_G}.
\end{equation}
For simplicity, we omit the cocharacters in the notation and write respectively $\HZip$ and $\GZip$ for the above stacks.

Write $\zeta_H\colon \overline{\Hscr}_{\mathbf{F},K'}\to \HZip$ for the map given by \eqref{zeta-Shimura}. For a fixed choice of compatible level structures $K'\subset \mathbf{H}(\AA_f^p)$ and $K\subset \mathbf{G}(\AA_f^p)$, we sometimes write $X_H\colonequals \overline{\Hscr}_{\mathbf{F},K'}$ and $X_G\colonequals \overline{\Ascr}_{n,K}$ in order to uniformize the notation. Write simply $\widetilde{u}$ for the map $\widetilde{u}_{K',K}\colon \overline{\Hscr}_{\mathbf{F},K'} \to \overline{\Ascr}_{n,K}$ given by \eqref{embed-HB-S}. By the construction of the maps $\zeta_H$ and $\zeta_G$, there is a commutative diagram:
\begin{equation}\label{diag-com-Shim}
\xymatrix@M=8pt{ 
X_H \ar[d]_{\zeta_H} \ar[r]^{\widetilde{u}} & X_G \ar[d]^{\zeta_G} \\
\HZip \ar[r]_{u_{\zip}} & \GZip.
}
\end{equation}

\subsection{The flag space}\label{subsec-flag}

\subsubsection{Flag strata}\label{sec-flag-strata}
We briefly consider the general setting before specializing to the groups $H,G$ of section \ref{sec-GZip-HT}. Let $(G,\mu)$ be a cocharacter datum and let $\Zcal_\mu=(G,P,Q,L,M)$ be the attached zip datum. Fix a Borel subgroup $B\subset P$. We defined in \cite[Part 1, \S2]{Goldring-Koskivirta-Strata-Hasse} the stack of $G$-zip flags $\GF^\mu$ (attached to $B$). It can be defined as the quotient stack
\begin{equation}\label{eq-GZipFlag}
    \GF^\mu = \left[ E \backslash \left(G_k\times P/B \right) \right].
\end{equation}
Here, $E$ acts on $G_k\times (P/B)$ by $(x,y)\cdot (g, hB) = (xgy^{-1}, xhB)$ for all $(x,y)\in E$, $g\in G$, $h\in P$. It also admits an interpretation as a moduli space of torsors (see \loccit Definition 2.1.1).
The first projection induces a morphism of stacks
\begin{equation}
    \pi \colon \GF^\mu \to \GZip^\mu
\end{equation}
whose fibers are isomorphic to $P/B$. Now, let $X$ be a $k$-scheme endowed with a morphism $\zeta\colon X\to \GZip^\mu$. We define the flag space $\Flag(X)$ of $X$ as the fiber product
\begin{equation}\label{zeta-flag}
    \xymatrix@1@M=5pt{
\Flag(X)\ar[r]^-{\zeta_{\flag}} \ar[d]_{\pi_X} & \GF^\mu \ar[d]^{\pi} \\
X \ar[r]_-{\zeta} & \GZip^\mu.
}
\end{equation}
We briefly recall the flag stratification of $\GF^\mu$, defined in \cite[\S 2.3]{Goldring-Koskivirta-Strata-Hasse}. Define $\Sbt\colonequals [B\backslash G_k / B]$ (we call this stack the Schubert stack of $G$). There is a natural map $\psi\colon \GF^\mu\to \Sbt$, defined as follows. The natural inclusion $G_k\to G_k\times P/B$, $x\mapsto (x,1)$ identifies $\GF^\mu$ with the quotient $[E'\backslash G_k]$, where $E'=E \cap (B\times G_k)$. One sees immediately that $E'\subset B\times {}^z B$ (where $z=\sigma(w_{0,I})w_0$ as defined in \eqref{z-def}), so that there is a canonical projection $[E'\backslash G_k]\to [B\backslash G_k / {}^z B]$. Finally, the map $G_k\to G_k$, $g\mapsto g \dot{z}$ yields an isomorphism $[B\backslash G_k / {}^z B]\to [B\backslash G_k / B]=\Sbt$. By composing these maps, we obtain a smooth, surjective morphism
\begin{equation}\label{psi-eq}
    \psi\colon \GF^\mu \to \Sbt.
\end{equation}
By the Bruhat stratification of $G$, the points of the underlying topological space of $\Sbt$ correspond bijectively to the elements of the Weyl group $W$. Specifically, the element $w\in W$ corresponds to the locally closed point 
\begin{equation}
    \Sbt_w\colonequals [B\backslash BwB /B].
\end{equation}
For $w\in W$, we denote by $\Fcal_{G,w}$ or simply $\Fcal_w$ the preimage $\psi^{-1}(\Sbt_w)$. It is locally closed in $\GF^\mu$. We call the locally closed subsets $(\Fcal_w)_{w\in W}$ the flag strata of $\GF^\mu$. Each flag stratum $\Fcal_w$ is smooth, and the Zariski closure $\overline{\Fcal}_w$ is normal. The image of a flag stratum under the projection $\pi\colon \GF^{\mu}\to \GZip^{\mu}$ is a union of zip strata (defined in section \ref{sec-param}). It is difficult to determine in general which zip strata appear in $\pi(\Fcal_w)$. However, in the special case $w\in {}^I W$, we have $\pi(\Fcal_w)=\Xcal_w$. Moreover, the map $\pi\colon \Fcal_w\to \Xcal_w$ is finite etale (\cite[Proposition 2.2.1]{Koskivirta-Normalization}). In particular, for $w_{\max}=w_{0,I}w_0$ (the longest element of ${}^I W$), we have $\pi(\Fcal_{w_{\max}})=\Ucal_\mu$ and the map $\pi\colon \Fcal_{w_{\max}}\to\Ucal_\mu$ is finite etale. For $w\in W$, we write
\begin{equation}\label{FlagXw-eq}
\Flag(X)_{w} \colonequals \zeta_{\flag}^{-1}(\Fcal_w)    
\end{equation}
for the corresponding flag stratum in $\Flag(X)$, where $\zeta_{\flag}$ is the map defined in diagram \eqref{zeta-flag}. If $\zeta$ is smooth, then each flag stratum $\Flag(X)_{w}$ is smooth and its Zariski closure $\overline{\Flag}(X)_{w}$ is normal.

\subsubsection{Functoriality}\label{sec-functo-flag}
The flag space satisfies functoriality properties, as explained in \cite[\S5.3]{Goldring-Koskivirta-zip-flags}. Let $(H,\mu_H)$ and $(G,\mu_G)$ be two cocharacter data and let $f\colon H\to G$ be a homomorphism over $\FF_p$ compatible with $\mu_H$ and $\mu_G$. Write $\Zcal_H=(H,P_H,Q_H,L_H,M_H)$ and $\Zcal_G=(G,P_G,Q_G,L_G,M_G)$ for the zip data associated with $(H,\mu_H)$ and $(G,\mu_G)$ respectively. Furthermore, assume that we have Borel subgroups $B_H\subset P_H$ and $B_G\subset P_G$ such that $f(B_H)\subset B_G$. Then, there is a natural morphism of $k$-stacks $f_{\flag} \colon \HF^{\mu_H}\to \GF^{\mu_G}$ which makes the diagram below commute
\begin{equation}\label{diag-flag-zip}
\xymatrix@1@M=9pt{
\HF^{\mu_H }\ar[d]_-{\pi_{H}} \ar[r]^{f_{\flag}} & \GF^{\mu_G} \ar[d]^{\pi_G} \\
\HZip^{\mu_H} \ar[r]_-{f_{\zip}} & \GZip^{\mu_G}.
}    
\end{equation}
Furthermore, the map $f_{\flag}$ sends a flag stratum of $\HF^{\mu_H}$ into a flag stratum of $\GF^{\mu}$ (indeed, the map $\psi$ defined in \eqref{psi-eq} is also functorial). Now, assume further that we have $k$-schemes $X_H$, $X_G$ endowed with morphisms $\zeta_H\colon X_H\to \HZip^{\mu_H}$ and $\zeta_G\colon X_G\to \GZip^{\mu_G}$, together with a morphism $\widetilde{f} \colon X_H\to X_H$ making the diagram below commute.
$$\xymatrix@1@M=9pt{
X_H \ar[d]_-{\zeta_{H}} \ar[r]^{\widetilde{f}} & X_G \ar[d]^{\zeta_G} \\
\HZip^{\mu_H} \ar[r]_-{f_{\zip}} & \GZip^{\mu_G}.
}$$
Then, by taking fiber products, we obtain a natural morphism $\widetilde{f}_{\flag}\colon \Flag(X_H)\to \Flag(X_G)$ such that the diagram below commutes.

\begin{center}
\begin{tikzcd}
        \Flag(X_H)\arrow[rr, "\zeta_{H,\flag}"]\arrow[dr, "\widetilde{f}_{\flag}"]\arrow[dd, "\pi_{X,H}", swap] & & \HF^{\mu_H} \arrow[dr, "f_{\flag}"]\arrow[dd, "\pi_H", near end] & \\
        & \Flag(X_G) \arrow[rr, "\zeta_{G,\flag}", crossing over, near start] & & \GF^{\mu_G} \arrow[dd, "\pi_G"] \\
        X_H \arrow[rr, "\zeta_{H}", near end]\arrow[dr, "\widetilde{f}"] & & \HZip^{\mu_H}\arrow[dr, "f_{\zip}"] & \\
        & X_G \arrow[rr, "\zeta_G"] \arrow[from=uu, "\pi_{X,G}", crossing over, near start, swap] & & \GZip^{\mu_G}.
    \end{tikzcd}    
\end{center}
We denoted by $\pi_H$, $\pi_G$, $\pi_{X,H}$, $\pi_{X,G}$, $\zeta_{H,\flag}$, $\zeta_{G,\flag}$ the obvious maps for each group. 

We now return to the setting of section \ref{sec-GZip-HT}. In particular, we write $X_H=\overline{\Hscr}_{\mathbf{F},K'}$, $X_G=\overline{\Ascr}_{n,K}$ for a compatible choice of level structures $K',K$, and we have the commutative diagram \eqref{diag-com-Shim}. Note that the parabolic subgroup attached to $\mu_H$ is the Borel subgroup $B_H$ of $H$. Therefore, the stacks $\HZip$ and $\HF$ coincide, and similarly we have $X_H=\Flag(X_H)$. Hence, the above diagram collapses in this case to the following, slightly simpler commutative diagram. To simplify the notation further, we write $\Flag_G\colonequals \Flag(X_G)$.

\begin{equation}\label{diag-HG}
\xymatrix@M=5pt@R=1.5cm@C=1.5cm{ 
X_H \ar@/_4.0pc/[dd]_{\widetilde{u}} \ar[r]^{\zeta_H} \ar[d]^{\widetilde{u}_{\flag}}  & \HZip \ar[d]_{u_{\flag}} \ar@/^4.0pc/[dd]^{u_{\zip}} \\
\Flag_G \ar[r]^{\zeta_{G,\flag}} \ar[d] &  \GF \ar[d] \\ 
X_G \ar[r]^{\zeta_G} & \GZip
}
\end{equation}
Let $w_{\max}=w_{0,I}w_0$ be the maximal element of ${}^I W$. The flag stratum $\Fcal_{G,w_{\max}}$ of $\GF$ will play an important role for us. By definition,  we have $\Fcal_{G,w_{\max}}=[E'\backslash B w_{\max} B z^{-1}]$. Concretely, the set $B w_{\max} B z^{-1}$ is the subset of $G_k$ of matrices of the form
\begin{equation}
    \left(
\begin{matrix}
A_1 & A_2 \\ A_3 & A_4 
\end{matrix}
    \right), \quad A_i\in M_n(k), \ \textrm{with } A_1\in \GL_n(k) \ \textrm{and lower-triangular.}
\end{equation}
The Zariski closure of $B w_{\max} B z^{-1}$ has a similar description, but the condition $A_1\in \GL_n(k)$ is removed.

\begin{proposition}\label{uflag-image}
The map \[u_{\flag} \colon \HZip\to \GF\]
maps the $\mu_H$-ordinary locus (i.e. the unique open zip stratum of $\HZip$) into the flag stratum $\Fcal_{G,w_{\max}}$ of $\GF$. In particular, the image of $u_{\flag}$ is contained in $\overline{\Fcal}_{G,w_{\max}}$.
\end{proposition}

\begin{proof}
The identity element $1\in H$ lies in the $\mu_H$-ordinary locus, and its image by $u$ is $1\in G$. By the above descritpion of $\Fcal_{w_{\max}}$, this point lies in $B w_{\max} B z^{-1}$. Since $u_{\flag}$ sends a flag stratum of $H$ into a flag stratum of $G$, we deduce that any point in the $\mu_H$-ordinary locus of $H$ is sent to a point of $\Fcal_{G,w_{\max}}$ by $u_{\flag}$.
\end{proof}

\section{\texorpdfstring{Automorphic forms in characteristic $p$}{}}

\subsection{Automorphic vector bundles} \label{aut-VB-sec}

We first review automorphic vector bundles on the generic fiber of Hodge-type Shimura varieties. We retain the notation of section \ref{sec-HT-SV}. Let $\mathbf{P}_{\pm}$ be the pair of opposite parabolic subgroups of $\mathbf{G}_\CC$ defined by our chosen cocharacter $\mu\colon \GG_{\mathrm{m},\CC}\to \mathbf{G}_\CC$, as in \eqref{Pminplus}, and write $\mathbf{P}\colonequals \mathbf{P}_-$. The Shimura variety $\Sh_K(\mathbf{G},\mathbf{X})\otimes_{\mathbf{E}}\CC$ carries a natural $\mathbf{P}$-torsor. Therefore, any algebraic representation $\rho\colon \mathbf{P}\to \GL(W)$ (where $W$ is a finite-dimensional $\CC$-vector space) gives rise to a vector bundle $\Vcal(\rho)$ on $\Sh_K(\mathbf{G},\mathbf{X})\otimes_{\mathbf{E}}\CC$. Let $\mathbf{B}\subset \mathbf{P}$ be a Borel subgroup. For a character $\lambda\in X^*(\mathbf{T})$, we may view $\lambda$ as a one-dimensional representation of $\mathbf{B}$ and consider the induced representation
\begin{equation}
    \mathbf{V}_I(\lambda)\colonequals \Ind_{\mathbf{B}}^{\mathbf{P}}(\lambda).
\end{equation}
We denote by $\Vcal_I(\lambda)$ the vector bundle attached to $\mathbf{V}_I(\lambda)$ and call it the automorphic vector bundle attached to $\lambda$. Note that $\mathbf{V}_I(\lambda)=0$ when $\lambda$ is not $I$-dominant (we say that a character $\lambda$ is $I$-dominant if $\langle \lambda, \alpha^\vee\rangle \geq 0$ for all $\alpha\in I$).

The vector bundle $\mathbf{V}_I(\lambda)$ admits an extension to the integral model $\Sscr_K$. Indeed, by our choice of $\mu$ (see section \ref{sec-GZip-HT}), we have a cocharacter $\mu\colon \GG_{\mathrm{m},\Ocal_{\mathbf{E}_\pfr}}\to \mathbf{G}_{\ZZ_p}\otimes_{\ZZ_p} \Ocal_{\mathbf{E}_\pfr}$. It induces therefore a parabolic subgroup $\Pscr\subset \mathbf{G}_{\ZZ_p}\otimes_{\ZZ_p} \Ocal_{\mathbf{E}_\pfr}$. By construction of the integral model $\Sscr_K$, there is a natural $\Pscr$-torsor on $\Sscr_K$. For each $\lambda\in X^*(\mathbf{T})$, consider the $\Pscr$-representation
\begin{equation}
\mathbf{V}_{I}(\lambda)_{\Pscr} \colonequals H^0(\Pscr/\Bscr,\Lscr_\lambda), 
\end{equation}
where $\Lscr_\lambda$ is the line bundle on $\mathscr{P}/\mathscr{B}$ naturally attached to $\lambda$ (strictly speaking, the line bundle $\Lscr_\lambda$ may not be defined over
$\Ocal_{\mathbf{E}_{\pfr}}$, so it may be necessary to pass to some ring extension contained in $\overline{\ZZ}_p$). By applying this representation to the $\Pscr$-torsor of $\Sscr_K$, we obtain a vector bundle $\Vcal_I(\lambda)$ on $\Sscr_K\otimes_{\Ocal_{\mathbf{E}_{\pfr}}} \overline{\ZZ}_p$ which extends the one on the generic fiber. We denote again its special fiber by $\Vcal_I(\lambda)$. It is a vector bundle on $S_K$ modeled on the $k$-representation $V_I(\lambda)=\Ind_B^P(\lambda)$, where $B$ and $P$ denote the special fibers of $\Bscr$ and $\Pscr$ respectively.

Let $\Omega$ be a field endowed with a ring homomorphism $\Ocal_{\mathbf{E}_{\pfr}}\to \Omega$. An element of the space
\begin{equation}
H^0(\Sscr_K\otimes_{\Ocal_{\mathbf{E}_\pfr}} \Omega, \Vcal_I(\lambda))
\end{equation}
will be called an automorphic form of weight $\lambda$ and level $K$ with coefficients in $\Omega$. In this paper, we are mainly interested in the case $\Omega=k$.

\subsection{\texorpdfstring{Vector bundles on $\GZip^\mu$}{}}

Let $(G,\mu)$ be a cocharacter datum with attached zip datum $\Zcal_{\mu}=(G,P,Q,L,M)$ and zip group $E$ (see \eqref{eq-Edef}). Let $\rho \colon E\to \GL(W)$ be an algebraic representation of $E$ on a finite-dimensional $k$-vector space $W$. By \cite[\S 2.4.2]{Imai-Koskivirta-vector-bundles}, we can attach a vector bundle $\Vcal(\rho)$ on $\GZip^\mu=[E\backslash G_k]$ using the associated sheaf construction of \cite[I.5.8]{jantzen-representations}. If $\rho \colon P\to \GL(W)$ is an algebraic representation of $P$, we view it is an $E$-representation via the first projection $\pr_1\colon E\to P$, and we denote again the associated vector bundle by $\Vcal(\rho)$. For $\lambda\in X^*(T)$, view $\lambda$ as a one-dimensional representation of $B$ and consider the induced $P$-representation
\begin{equation}
V_I(\lambda) \colonequals \Ind_B^P(\lambda).
\end{equation}
The vector bundle on $\GZip^\mu$ attached to $V_I(\lambda)$ is denoted by $\Vcal_I(\lambda)$. The vector bundles $\Vcal_I(\lambda)$ on $\GZip^\mu$ and $S_K$ are compatible via the map $\zeta\colon S_K\to \GZip^\mu$, in the sense that $\zeta^*(\Vcal_I(\lambda))$ coincides with the vector bundle $\Vcal_I(\lambda)$ defined in section \ref{aut-VB-sec}.

By a similar construction, each algebraic $B$-representation $\rho\colon B\to \GL(W)$ gives rise to a vector bundle $\Vcal_{\flag}(\rho)$ on the flag space $\GF^\mu$ (\cite[\S 3]{Imai-Koskivirta-partial-Hasse}). In particular, we have a family of line bundles $\Vcal_{\flag}(\lambda)$ for $\lambda\in X^*(T)$. We denote again by $\Vcal_{\flag}(\lambda)$ the line bundle $\zeta_{\flag}^*(\Vcal_{\flag}(\lambda))$ on the flag space $\Flag(S_K)$. For a $B$-representation $\rho$, the push-forward of $\Vcal_{\flag}(\rho)$ via the map $\pi\colon \GF^\mu \to \GZip^\mu$ is given by
\[
\pi_*(\Vcal_{\flag}(\rho))=\Vcal(\Ind_B^P(\rho))
\]
by \cite[Proposition 3.2.1]{Imai-Koskivirta-partial-Hasse}. In particular, when $\rho$ is the one-dimensional $B$-representation given by a character $\lambda \in X^*(T)$, we have the identification 
\begin{equation}\label{Vcal-formula}
    \pi_*(\Vcal_{\flag}(\lambda))=\Vcal_I(\lambda).
\end{equation}
This yields an identification $H^0(\GZip^\mu,\Vcal_I(\lambda))=H^0(\GF^\mu,\Vcal_{\flag}(\lambda))$. For a general pair $(G,\mu)$, this space can be described as the part of the Brylinski--Kostant filtration of $V_I(\lambda)$ stable under the action of a certain finite group (see \cite[Theorem 3.4.1]{Imai-Koskivirta-vector-bundles}). When $P$ is defined over $\FF_p$, it has a simpler description: For each $\chi\in X^*(T)$, write $V_I(\lambda)_\chi$ for the $\chi$-weight space of $V_I(\lambda)$ and define
\begin{equation}\label{H0zip-eq}
  V_I(\lambda)_{\leq 0} = \bigoplus_{\substack{\langle \chi,\alpha^\vee\rangle \leq 0 \\
  \textrm{for all} \ \alpha\in \Delta\setminus I}}
  V_I(\lambda)_{\chi}.
\end{equation}
Then we have an identification (\cite[Theorem 3.7.2]{Koskivirta-automforms-GZip})
\begin{equation}
H^0(\GZip^\mu,\Vcal_I(\lambda))=V_I(\lambda)_{\leq 0} \cap V_I(\lambda)^{L(\FF_p)}.
\end{equation}
Assume now that $S_K$ is the good reduction special fiber of a Hodge-type Shimura variety as in section \ref{sec-GZip-HT}, and $\GZip^\mu$ is the attached stack of $G$-zips. Denote by $\pi_K$ the projection $\pi_K\colon \Flag(S_K)\to S_K$ (see diagram \eqref{zeta-flag}). Then formula \eqref{Vcal-formula} holds also on the level of $S_K$, i.e. one has $\pi_{K,*}(\Vcal_{\flag}(\lambda))=\Vcal_I(\lambda)$ (indeed, since $\zeta$ is smooth, push-forward and pull-back commute). Therefore, the space of automorphic forms of weight $\lambda$ can be identified with
\begin{equation} \label{identif-lambda}
H^0(S_K,\Vcal_I(\lambda))=H^0(\Flag(S_K),\Vcal_{\flag}(\lambda)).
\end{equation}
We are thus reduced to studying the line bundle $\Vcal_{\flag}(\lambda)$ on the flag space of $S_K$.

\subsection{Extension to the toroidal compactification}\label{sec-tor-ext}

Retain the notations and assumptions of section \ref{sec-GZip-HT}. By \cite[Theorem 1]{Madapusi-Hodge-Tor}, we can find a sufficiently fine cone decomposition $\Sigma$ and a smooth toroidal compactification $\Sscr_K^\Sigma$ of $\Sscr_K$ over $\Ocal_{\mathbf{E}_\pfr}$. The family $(\Vcal_I(\lambda))_{\lambda\in X^*(\mathbf{T})}$ admits a canonical extension $(\Vcal^{\Sigma}_I(\lambda))_{\lambda\in X^*(\mathbf{T})}$ to $\Sscr^{\Sigma}_K$. The following result is known as the Koecher principle. 

\begin{theorem}[{\cite[Theorem 2.5.11]{Lan-Stroh-stratifications-compactifications}}] \label{thm-koecher}
Let $\Omega$ be a field which is an $\Ocal_{\mathbf{E}_{\pfr}}$-algebra, and assume that $\Vcal_I(\lambda)$ is defined over $\Omega$. The natural map
\begin{equation}
H^0(\Sscr^{\Sigma}_K\otimes_{\Ocal_{\mathbf{E}_{\pfr}}} \Omega,\Vcal^{\Sigma}_I(\lambda)) \to H^0(\Sscr_K\otimes_{\Ocal_{\mathbf{E}_{\pfr}}} \Omega,\Vcal_I(\lambda)) 
\end{equation}
is a bijection, except when $\dim(\Sh_K(\mathbf{G},\mathbf{X}))=1$ and $\Sscr^{\Sigma}_K \setminus \Sscr_K\neq \emptyset$.
\end{theorem}
The Koecher principle holds for the Shimura varieties considered in this article. Consider the special fiber $S_K^\Sigma \colonequals \Sscr_K^{\Sigma}\otimes_{\Ocal_{\mathbf{E}_{\pfr}}} k$. By \cite[Theorem 6.2.1]{Goldring-Koskivirta-Strata-Hasse}, the map $\zeta\colon S_{K}\to \GZip^\mu$ extends naturally to a map of stacks
\begin{equation}
    \zeta^\Sigma\colon S_{K}^{\Sigma}\to \GZip^\mu
\end{equation}
and by \cite[Theorem 1.2]{Andreatta-modp-period-maps}, the extended map $\zeta^\Sigma$ is again smooth. Since $\zeta$ is surjective, $\zeta^\Sigma$ is also surjective. By construction, the pull-back by $\zeta^{\Sigma}$ of $\Vcal_I(\lambda)$ coincides with the canonical extension $\Vcal^{\Sigma}_I(\lambda)$.

\subsection{Finite etale maps}

Let $f\colon Y\to X$ be a finite etale map of degree $n\geq 1$ between noetherian schemes. If $\Lscr$ is a line bundle on $X$, it is known that one has a natural trace map
\begin{equation}
    \Tr\colon H^0(Y,f^*(\Lscr))\to H^0(X,\Lscr).
\end{equation}
We generalize this construction to other symmetric functions. We first consider the case when $\Lscr=\Ocal_X$ and define maps
\begin{equation}
    c_i\colon H^0(Y,\Ocal_Y)\to H^0(X,\Ocal_X)
\end{equation}
for each $0\leq i \leq n-1$ as follows. We first examine the case when $X=\spec(A)$, $Y=\spec(B)$ and $B$ is a free $A$-algebra. In this case, for each $b\in B$, we may consider the characteristic polynomial $\chi_b(X)$ of the $A$-linear map $B\to B$, $x\mapsto bx$. Then, define $c_i\colon B\to A$ as the $i$-th coefficient of $\chi_b(X)$, for each $0\leq i\leq n-1$. Next, assume that $B$ is a finite etale $A$-algebra. In this case $B$ is a locally free $A$-algebra, so by taking an appropriate open covering of $A$, we may extend the definition of $\chi_b(X)$ and $c_i(b)$ by glueing. This defines a map of sets $c_i\colon B\to A$. Similarly, for an arbitrary finite etale map $f\colon Y\to X$ of degree $n$, we choose an open affine covering of $X$ and glue the maps $c_i$ defined in the affine case. We obtain a map $c_i\colon H^0(Y,\Ocal_Y)\to H^0(X,\Ocal_X)$. Next, we consider an arbitrary line bundle $\Lscr$ on $X$. First assume that $X=\spec(A)$, $Y=\spec(B)$ where $B$ is a free $A$-algebra, and $M$ is a free $A$-module of rank one. Choose any element $m\in M$ such that $M=Am$ and define a map (of sets) by
\begin{equation}
c_{i,M}\colon B\otimes_A M \to M^{\otimes i}, \quad b\otimes (am) \mapsto c_i(b) (am)\otimes \dots \otimes (am)
\end{equation}
This map is well-defined since $c_i(ab)=a^i c_i(b)$ for any $a\in A$. By glueing, we may define $c_{i,M}$ when $M$ is a locally free $A$-module of rank one, and eventually we obtain a map
\begin{equation}
    c_{i,\Lscr}\colon H^0(Y,f^*(\Lscr))\to H^0(X,\Lscr^{\otimes i})
\end{equation}
in the general setting. By our construction, the case $i=1$ coincides with the usual trace map.

\subsection{Stable base locus}

Let $\Lscr$ be a line bundle on a scheme $X$. Recall that the base locus of $\Lscr$ is the set of points $x\in X$ such that any section $s\in H^0(X,\Lscr)$ vanishes at $x$. Denote by $B(\Lscr)$ the base locus of $\Lscr$ and set:
\begin{equation}
    \BB(\Lscr)\colonequals \bigcap_{d\geq 1} B(\Lscr^{\otimes d}).
\end{equation}
The set $\BB(\Lscr)$ is called the stable base locus of $\Lscr$. Clearly, it is a Zariski closed subset of $X$. If $f\colon Y\to X$ is a morphism of schemes and $\Lscr$ is a line bundle on $X$, then it is clear that $f(B(f^*(\Lscr)))\subset B(\Lscr)$. Similarly, we have an inclusion $f(\BB (f^*(\Lscr)))\subset \BB(\Lscr)$, in other words
\begin{equation}
\BB (f^*(\Lscr))\subset f^{-1}(\BB(\Lscr)).
\end{equation}

\begin{proposition}\label{prop-fineet-base}
Let $K$ be an algebraically closed field and let $X,Y$ be $K$-varieties. Assume that $f\colon Y\to X$ is a finite etale map. For any line bundle $\Lscr$ on $X$, we have 
\begin{equation}
\BB (f^*(\Lscr)) = f^{-1}(\BB(\Lscr)).
\end{equation}
\end{proposition}

\begin{proof}
Let $y\in Y$ and $x\colonequals f(y)\in X$ such that $x\in \BB(\Lscr)$. We need to show that $y\in \BB(f^*(\Lscr))$. Let $s\in H^0(Y,f^*(\Lscr)^{\otimes d})$ be a global section. We will show that $s$ vanishes at each point of $f^{-1}(x)$. By assumption, for any $0\leq i \leq n-1$ (where $n=\deg(f)$), the section $c_i(s)\in H^0(X,\Lscr^{\otimes id})$ vanishes at $x$. Let $U=\spec(A)$ be an open affine neighborhood of $x$ and write $f^{-1}(U)=\spec(B)$. We may assume that $B$ is a free $A$-module and that $\Lscr$ is trivial on $U$. Since $K$ is algebraically closed, it suffices to consider closed points, hence we may assume that $x$ corresponds to a maximal ideal $\mfr\subset A$. The fiber $f^{-1}(x)$ consists of maximal ideals $\mfr_1, \dots, \mfr_n$ of $B$. Taking tensor product with $A/\mfr$, we obtain a Cartesian diagram
\begin{equation}
\xymatrix@M=5pt{
B \ar[r]^{c_i} \ar[d] & A \ar[d] \\
B/\mfr B \ar[r]^{c_i} & A/\mfr
}
\end{equation}
Since $K$ is algebraically closed, we have $ A/\mfr=K$. The inclusion $\mfr B\subset \bigcap_{i=1}^n \mfr_i = \prod_{i=1}^n \mfr_i$, yields a natural projection 
\begin{equation}
B/\mfr B\to B/\prod_{i=1}^n \mfr_i = \prod_{i=1}^n B/\mfr_i = K^n.    
\end{equation}
Since $B$ is a free $A$-module of rank $n$, $B/\mfr B$ has dimension $n$ over $K$, so we deduce that this projection map is an isomorphism. In particular, the lower horizontal map in the above Cartesian diagram identifies with the map
\begin{equation}
   c_i \colon  K^n\to K, \quad (x_1,\dots,x_n)\mapsto c_i(x_1,\dots,x_n)
\end{equation}
where $c_i(x_1,\dots,x_n)$ is the coefficient of degree $i$ in the polynomial
\begin{equation}
    (X-x_1)(X-x_2)\dots (X-x_n)
\end{equation}
Write $z_i\colonequals s + \mfr_i \in K$ for the value of $s$ at the point $\mfr_i\in \spec(B)$. By assumption, we have $c_i(z_1,\dots,z_n)=0$ for all $0\leq i \leq n-1$. It follows that $(X-z_1)(X-z_2)\dots (X-z_n) = X^n$, and therefore $z_i=0$ for all $i$. This terminates the proof.
\end{proof}

\subsection{Hecke-equivariance and automorphic vector bundles}

We apply Proposition \ref{prop-fineet-base} in the context of Shimura varieties. First, we need to define Hecke operators on the level of the flag space $\Flag(S_K)$. For each compact open subgroup $K^p\subset \mathbf{G}(\AA_f^p)$ and each $g\in \mathbf{G}(\AA_f^p)$, recall that we have the Hecke correspondance given by diagram \eqref{hecke-diag}. Passing to the special fiber and taking the fiber product with $\GF^\mu$ over $\GZip^\mu$, we obtain a diagram
\begin{equation}
    \xymatrix{
& \Flag(S_{K\cap g K^p g^{-1}}) \ar[dl]_{\pi_{\flag}}  \ar[dr]^{g_{\flag}} & \\
\Flag(S_{K}) &  & \Flag(S_{K})
}
\end{equation}
with finite etale maps. We deduce from this that for any $\lambda\in X^*(\mathbf{T})$, there is a natural action of the Hecke algebra $\TT_{K^p}\colonequals \TT_{K^p}(\mathbf{G})$ on the space $H^0(\Flag(S_{K}),\Vcal_{\flag}(\lambda))$. Similarly, the above diagram is obviously compatible with the flag stratifications, so we obtain actions of $\TT_{K^p}$ on the spaces
\begin{equation}
 H^0(\Flag(S_{K})_w,\Vcal_{\flag}(\lambda)) \quad \textrm{and} \quad H^0(\overline{\Flag}(S_{K})_w,\Vcal_{\flag}(\lambda))   
\end{equation}
for any $w\in W$. Furthermore, $\TT_{K^p}$ acts by correspondences on the 0-cycles of $\Flag(S_{K})_w$ and $\overline{\Flag}(S_{K})_w$. For $x\in \Flag(S_{K})$ and $g\in \mathbf{G}(\AA_f^p)$, we define a 0-cycle
\begin{equation}\label{TKp-flag}
    T_{K^p,g}(x) = g_{\flag, *} \pi_{\flag}^*(x)
\end{equation}
similarly to the case of $S_{K}$ (see \eqref{TKp}). We view equation \eqref{TKp-flag} as a formal sum of points of $\Flag(S_{K})$.

\begin{definition}
Let $A$ be a subset of $S_{K}$ (resp. $\Flag(S_{K})$). We say that $A$ is stable by Hecke operators if for any $x\in A$ and any $g\in \mathbf{G}(\AA_f^p)$, all points in the formal sum $T_{K^p,g}(x)$ lie in $A$.
\end{definition}

Fix $w\in W$ and $\lambda\in X^*(T)$. Write
\begin{equation}
    \BB_{K^p,w}(\lambda)
\end{equation}
for the stable base locus of the line bundle $\Vcal_{\flag}(\lambda)$ restricted to  $\overline{\Flag}(S_{K})_w$ (the Zariski closure of $\Flag(S_{K})_w$, endowed with the reduced structure). By Proposition \ref{prop-fineet-base}, for any inclusion $K'^{p}\subset K^p$ of compact open subgroups of $\mathbf{G}(\AA_f^p)$, we have
\begin{equation}\label{binfty-Kp}
    \pi_{\flag}^{-1}( \BB_{K^p,w}(\lambda)) =  \BB_{K'^{p},w}(\lambda)
\end{equation}
where $\pi=\pi_{K',K}$ denotes the change of level map with respect to $K'^p\subset K^p$. We deduce the following result:

\begin{theorem}\label{B-hecke-stable}
The set $\BB_{K^p,w}(\lambda)$ is stable by Hecke operators.
\end{theorem}

\begin{proof}
Let $x\in \BB_{K^p,w}(\lambda)$ and $g\in \mathbf{G}(\AA_f^p)$. By \eqref{binfty-Kp}, any point in $\pi_{\flag}^{-1}(x)$ lies in the set $\BB_{K\cap gKg^{-1},w}(\lambda)$. Therefore, $g_{\flag}(\pi_{\flag}^{-1}(x))$ is contained in $\BB_{K^p,w}(\lambda)$.
\end{proof}

In the case $w=w_{\max}$, this theorem has the following consequence:

\begin{corollary}\label{cor-sbl-zero}
Assume that $\mathbf{E}_{\pfr}=\QQ_p$. Let $\lambda\in X^*(T)$ and assume that 
\begin{equation}
    \BB_{K^p,w_{\max}}(\lambda) \neq \emptyset.
\end{equation}
Then $\BB_{K^p,w_{\max}}(\lambda)=\overline{\Flag}(S_{K})_{w_{\max}}$. In other words, for any $r\geq 1$, one has 
\begin{equation}
    H^0(\overline{\Flag}(S_{K})_{w_{\max}}, \Vcal_{\flag}(\lambda)^{\otimes r}) = 0.
\end{equation}
\end{corollary}

\begin{proof}
Since $\mathbf{E}_{\pfr}=\QQ_p$, we may use Theorem \ref{thm-Hx-dense}. Choose $x\in \BB_{K^p,w_{\max}}(\lambda)$ and consider the prime-to-$p$ Hecke orbit $\Hcal^p(x)\subset \Flag(S_{K})_{w_{\max}}$. By Theorem \ref{B-hecke-stable}, $\Hcal^p(x)$ is contained in the stable base locus $\BB_{K^p,w_{\max}}(\lambda)$. If we write $x'\colonequals \pi(x)\in S_{K}^{\rm ord}$, the Hecke orbit $\Hcal^p(x')$ is Zariski dense in $S_{K}$ by Theorem \ref{thm-Hx-dense}. By Lemma \ref{lemma-dense} below, we deduce that $\Hcal^p(x)$ is dense in $\Flag(S_{K})_{w_{\max}}$. Since $\BB_{K^p,w_{\max}}(\lambda)$ is closed, it follows that $\BB_{K^p,w_{\max}}(\lambda)=\overline{\Flag}(S_{K})_{w_{\max}}$.
\end{proof}

\begin{lemma}\label{lemma-dense}
    Let $A\subset \Flag(S_{K})_{w_{\max}} $ be a subset such that $\pi(A)$ is Zariski dense in $S_{K}$. Then $A$ is Zariski dense in $\Flag(S_{K})_{w_{\max}}$.
\end{lemma}
\begin{proof}
We have a Cartesian diagram
\begin{equation}
    \xymatrix{
\Flag(S_{K})_{w_{\max}} \ar[r] \ar[d] & \Fcal_{w_{\max}} \ar[d] \\
S_{K, w_{\max}} \ar[r]_{\zeta} & \Xcal_{w_{\max}}
}
\end{equation}
Here, $\Xcal_{w_{\max}}=\Ucal_\mu$ denotes the unique open zip stratum of $\GZip^\mu$ (see section \ref{sec-param}) and $S_{K,w_{\max}}=S^{\ord}_K$ is the ordinary locus (see section \ref{sec-GZip-HT}). Furthermore, the vertical maps are surjective, finite etale (as explained in section \ref{sec-flag-strata}). The stacks $\Fcal_{w_{\max}}$ and $\Xcal_{w_{\max}}$ are connected, so the above diagram induces a bijection between connected components of $\Flag(S_{K})_{w_{\max}}$ and connected components of $S_{K}$. By the above, the Zariski closure $\overline{A}$ of $A$ intersected with any connected component of $\Flag(S_{K})_{w_{\max}}$ has dimension $\dim(S_{K})=\dim(\Flag(S_{K})_{w_{\max}})$. Since $\Flag(S_{K})_{w_{\max}}$ is smooth, each connected component is irreducible, so we deduce that $A$ is dense in $\Flag(S_{K})_{w_{\max}}$.
\end{proof}

\section{The cone conjecture and consequences}

\subsection{Statement}

The cone conjecture was first formulated in \cite{Goldring-Koskivirta-global-sections-compositio}. We can state it in the broader context of a general $k$-scheme $X$ endowed with a morphism $\zeta\colon X\to \GZip^\mu$. For such a pair $(X,\zeta)$, define $\Flag(X)$ as in \S\ref{sec-flag-strata}. For any $w\in {}^I W$, write $X_w\colonequals \zeta^{-1}(\Xcal_w)$ for the corresponding locally closed subset of $X$. Similarly, for $w\in W$ define $\Flag(X)_w$ as in equation \eqref{FlagXw-eq}. We make the following assumptions on $(X,\zeta)$:

\begin{assumption}\label{assume-conj} \ 
\begin{assertionlist}
\item $\zeta$ is smooth.
\item The restriction of $\zeta$ to any connected component of $X$ is surjective.
\item For any $w\in W$ of length $1$, the Zariski closure of $\Flag(X)_w$ is pseudo-complete.
\end{assertionlist}
\end{assumption}

In assumption (3) above, a $k$-scheme $Y$ is called pseudo-complete if any element of $\Ocal_Y(Y)$ is locally constant. For example, (3) is satisfied if $X$ is a proper $k$-scheme. Since $\GZip^\mu$ is smooth over $k$, (1) implies that $X$ is a smooth $k$-scheme. The pair $(S_K^\Sigma,\zeta^\Sigma)$ defined in section \ref{sec-tor-ext} satisfies Assumption \ref{assume-conj}, by \cite[\S 1.1.4]{Goldring-Koskivirta-GS-cone}. For $\lambda\in X^*(T)$, the pullback $\zeta^*(\Vcal_I(\lambda))$ is a vector bundle on $X$ that we denote again by $\Vcal_I(\lambda)$. Define the cone of $X$ as follows:
\begin{equation}
    C_X\colonequals \{\lambda\in X^*(T) \ | \ H^0(X,\Vcal_I(\lambda))\neq 0 \}.
\end{equation}
When $X$ is connected, one can check immediately that $C_X$ is a cone in $X^*(T)$ (i.e. an additive monoid). Write $X^*_{+,I}(T)$ for the set of $I$-dominant characters. Since $\Vcal_I(\lambda)=0$ when $\lambda$ is not $I$-dominant, we obviously have $C_X\subset X^*_{+,I}(T)$. We define similarly
\begin{equation}
    C_{\zip}\colonequals \{\lambda\in X^*(T) \ | \ H^0(\GZip^\mu,\Vcal_I(\lambda))\neq 0 \}
\end{equation}
and call $C_{\zip}$ the zip cone of $(G,\mu)$. Using the description of $H^0(\GZip^\mu,\Vcal_I(\lambda))$ given by \eqref{H0zip-eq}, the cone $C_{\zip}$ relates to the representation theory of reductive groups. Since $\zeta$ is surjective, we obtain by pull-back an injective map
\begin{equation}
  \zeta^* \colon  H^0(\GZip^\mu,\Vcal_I(\lambda)) \to H^0(X,\Vcal_I(\lambda)).
\end{equation}
This shows that we have an inclusion $C_{\zip}\subset C_X$. For a subset $C\subset X^*(T)$, we define the saturation of $C$ by:
\begin{equation}
    \Ccal \colonequals \{\lambda\in X^*(T) \ | \ \exists N\geq 1 , \ N\lambda\in C\}.
\end{equation}
If $C$ is a cone of $X^*(T)$, then so is its saturation. Write $\Ccal_{\zip}$ and $\Ccal_X$ for the saturations of $C_{\zip}$ and $C_X$ respectively.

\begin{conjecture}\label{cone-conj}
If $(X,\zeta)$ satisfies Assumption \ref{assume-conj}, we have $\Ccal_X = \Ccal_{\zip}$.
\end{conjecture}

We retain the notation from \S\ref{sec-GZip-HT} and consider the Shimura variety $\Sscr_K$. For a field $\Omega$ endowed with a map $\Ocal_{\mathbf{E}_{\pfr}}\to \Omega$, put 
\begin{equation}
    C_K(\Omega)\colonequals \{\lambda\in X^*(T) \ | \ H^0(\Sscr_K \otimes_{\Ocal_{\mathbf{E}_{\pfr}}} \Omega, \Vcal_I(\lambda))\neq 0 \}.
\end{equation}
This set highly depends on the level $K$, but the saturation $\Ccal(\Omega)$ is independent of $K$ by \cite[Corollary 1.5.3]{Koskivirta-automforms-GZip}. Goldring and the author have started a program in \cite{Goldring-Koskivirta-global-sections-compositio,Goldring-Koskivirta-GS-cone} to determine $\Ccal(k)$ (where $k=\overline{\FF}_p$ is an algebraic closure of $\FF_p$). By the Koecher principle (Theorem \ref{thm-koecher}), the global sections of $\Vcal_I(\lambda)$ over $S_K$ and those of $\Vcal^{\Sigma}_I(\lambda)$ over $S_K^\Sigma$ coincide. Since the pair $(S_K^\Sigma,\zeta^\Sigma)$ satisfies Assumption \ref{assume-conj}, Conjecture \ref{cone-conj} can be applied to the pair $(S_K^\Sigma,\zeta^\Sigma)$ and predicts that $\Ccal(k)= \Ccal_{\zip}$.

The above conjecture is a generalization of a result of Diamond--Kassaei (\cite[Corollary 8.2]{Diamond-Kassaei}) stating that the weight of any nonzero Hilbert modular form in characteristic $p$ is spanned by the weight of partial Hasse invariants (this was proved independently in \cite{Goldring-Koskivirta-global-sections-compositio}). This result shows that the Ekedahl--Oort stratification (and hence the stack $\GZip^\mu$) encodes information about the weights of all automorphic forms. We have verified Conjecture \ref{cone-conj} in \cite{Goldring-Koskivirta-global-sections-compositio}, \cite{Goldring-Koskivirta-divisibility} for several pairs $(G,\mu)$:
\begin{theorem}[\cite{Goldring-Koskivirta-global-sections-compositio, Goldring-Koskivirta-divisibility, Koskivirta-Hilbert-strata}]
Conjecture \ref{cone-conj} holds in the following cases. In cases \textup{(b)} through \textup{(f)}, assume that $\mu$ is a minuscule cocharacter of $G$.
\begin{definitionlist}
\item $G$ is an $\FF_p$-form of $(\GL_{2,k})^n$.
\item $G=\GL_{3,\FF_p}$ or $G=\GL_{4,\FF_p}$,
\item $G=\GSp_{4,\FF_p}$,
    \item $G=\GSp_{6,\FF_p}$ and $p\geq 5$,
    \item $G=\GU(3)_{\FF_p}$,
    \item $G=\GU(4)_{\FF_p}$ and the type of $\mu$ is not $(2,2)$.
\end{definitionlist}
\end{theorem}

In particular, the above result applies to Hilbert modular varieties, $A_1$-type Shimura varieties (see \cite[\S 2.5]{Koskivirta-Hilbert-strata} for the terminology), several unitary Shimura varieties at split or inert primes of good reduction, and Siegel-type Shimura varieties in rank $\leq 3$.

\subsection{\texorpdfstring{Approximations of $C_{\zip}$}{}}

In general, it is a difficult problem to determine the zip cone $C_{\zip}$ or even its saturation $\Ccal_{\zip}$. In \cite{Imai-Koskivirta-zip-schubert}, we defined several subcones of $C_{\zip}$ that give first approximations. For example, define the Griffiths--Schmid cone $\Ccal_{\GS}$ by
\begin{equation}\label{CGS-eq}
\Ccal_{\GS}=\left\{ \lambda\in X^{*}(\mathbf{T}) \ \relmiddle| \ 
\parbox{6cm}{
$\langle \lambda, \alpha^\vee \rangle \geq 0 \ \textrm{ for }\alpha\in I, \\
\langle \lambda, \alpha^\vee \rangle \leq 0 \ \textrm{ for }\alpha\in \Phi^+ \setminus \Phi^+_{\mathbf{L}}$}
\right\}.
\end{equation}
The cone $\Ccal_{\GS}$ is expected to coincide with $\Ccal(\CC)$ (this seems to be known to experts, even though no explicit proof can be found in the literature; the inclusion $\Ccal(\CC)\subset \Ccal_{\GS}$ is proved explicitly in \cite{Goldring-Koskivirta-GS-cone}). Furthermore, by a mod $p$ reduction argument (\cite[Proposition 1.8.3]{Koskivirta-automforms-GZip}), we have an inclusion $C_K(\CC)\subset C_K(\overline{\FF}_p)$, and hence also $\Ccal(\CC)\subset \Ccal(\overline{\FF}_p)$. Therefore, combining Conjecture \ref{cone-conj} with the equality $\Ccal_{\GS}=\Ccal(\CC)$, we deduce that one should expect an inclusion $\Ccal_{\GS}\subset \Ccal_{\zip}$. We showed in \cite[Theorem 6.4.2]{Imai-Koskivirta-zip-schubert} that this containment always holds true, confirming this prediction. Apart from $\Ccal_{\GS}$, other natural subcones of $C_{\zip}$ are defined in \cite{Imai-Koskivirta-zip-schubert}: The cone of partial Hasse invariants $C_{\Hasse}$, the highest weight cone $C_{\hw}$, the lowest weight cone $C_{\lw}$, etc. We omit the definitions of these cones here.

The above subcones give lower bounds for the sets $\Ccal_{\zip}$ and $\Ccal(\overline{\FF}_p)$ but they do not provide an upper bound, i.e. an explicit condition on $\lambda\in X^*(T)$ for the space $H^0(S_K,\Vcal_I(\lambda))$ to vanish. In the paper \cite{Goldring-Koskivirta-GS-cone}, we determined an upper bound for $\Ccal_{\zip}$. We showed that the cone $\Ccal_{\zip}$ is contained in a (rather) explicit cone, called the unipotent-invariance cone. We omit the general definition, which is somewhat technical. The situation simplifies greatly if we make the following assumptions:
\begin{assertionlist}
\item The parabolic $P$ is defined over $\FF_{p}$.
\item The group $G$ is split over $\FF_{p^2}$.
\end{assertionlist}
In this case, one can give a quite sharp and explicit approximation from above for the cone $\Ccal_{\zip}$. Let $W_L=W(L,T)$ be the Weyl group of $L$. Note that $W_L \rtimes \Gal(\FF_{p^2}/\FF_p)$ acts naturally on the set $\Phi^+ \setminus \Phi^+_{L}$. We proved (\cite[Theorem 3.2.5]{Goldring-Koskivirta-GS-cone} and \S3.3):

\begin{proposition}
Any $\lambda\in \Ccal_{\zip}$ satisfies 
\begin{equation}\label{orb-ineq}
    \sum_{\alpha \in \Ocal\setminus S} \langle \lambda,\alpha^\vee \rangle \ + \ \frac{1}{p} \sum_{\substack{\alpha\in S}} \langle  \lambda, \alpha^\vee \rangle \leq 0
\end{equation}
for all $W_{L}\rtimes \Gal(\FF_{p^2}/\FF_p)$-orbits $\Ocal\subset \Phi^+\setminus \Phi^+_{L}$ and all subsets $S\subset \Ocal$. 
\end{proposition}

Following the notation and terminology of \cite{Goldring-Koskivirta-GS-cone}, we denote by $\Ccal_{L-\Min}$ the set of $\lambda\in X^*(T)$ satisfying the  above inequalities \eqref{orb-ineq}, and call $\Ccal_{L-\Min}$ the $L$-minimal cone. We also define
\begin{equation}\label{LMinplusI}
    \Ccal_{L-\Min}^{+,I}\colonequals \Ccal_{L-\Min}\cap X^*_{+,I}(T).
\end{equation}
Since $\Ccal_{\zip}$ and $\Ccal(\overline{\FF}_p)$ are always contained in the $I$-dominant cone $X^*_{+,I}(T)$, the cone $\Ccal_{L-\Min}^{+,I}$ is actually more relevant to us. The number of inequalities necessary to define $\Ccal_{L-\Min}^{+,I}$ is usually significantly smaller than for $\Ccal_{L-\Min}$, but it is difficult to determine a set of minimal inequalities for it. In particular, Conjecture \ref{cone-conj} implies the following:

\begin{conjecture}\label{conj-orb-ineq}
Let $S_K$ be the special fiber of a Hodge-type Shimura variety at a place $p$ of good reduction. Assume that $p$ splits in the reflex field $\mathbf{E}$ and that the attached reductive $\FF_p$-group $G$ is split over $\FF_{p^2}$. Then the space $H^0(S_K,\Vcal_I(\lambda))$ is zero outside of the locus defined by the inequalities \eqref{orb-ineq}.
\end{conjecture}

\subsection{Partial Hasse invariants} \label{pha-sec}

For $w\in W$, define a set $E_w$ as follows:
\begin{equation}\label{def-Ew}
    E_w\colonequals \{\alpha\in \Phi^+ \ | \ w s_{\alpha}< w \ \textrm{and} \ \ell(ws_\alpha)=\ell(w)-1 \}.
\end{equation}
The elements $ws_\alpha$ for $\alpha\in E_w$ are precisely the "lower neighbors" of $w$ in $W$ with respect to the Bruhat order. For a pair $(\chi,\eta)\in X^*(T)^2$, let $\Vcal_{\Sbt}(\chi,\eta)$ be the attached line bundle on the Schubert stack $\Sbt$. We review Chevalley's formula for the stratum $\Sbt_w$. The strata contained in the Zariski closure $\overline{\Sbt}_w$ with codimension $1$ are precisely those of the form $\Sbt_{ws_\alpha}$ for $\alpha\in E_w$.

\begin{theorem}[{\cite[Theorem 2.2.1]{Goldring-Koskivirta-Strata-Hasse}}]\label{brion}
Let $w \in W$. One has the following:
\begin{assertionlist}
\item \label{brion1} $H^0\left(\Sch_w,\Vcal_{\Sch}(\chi,\eta)\right)\neq 0\Longleftrightarrow \eta = -w^{-1} \chi$.
\item \label{brion2} $\dim_k H^0\left(\Sch_w,\Vcal_{\Sch}(\chi,-w^{-1} \chi) \right)=1$.
\item \label{brion3} For any nonzero $f\in H^0\left(\Sch_w,\Vcal_{\Sch}(\chi,-w^{-1} \chi) \right)$, one has
\begin{equation}\label{briondiv}
\div(f)=-\sum_{\alpha \in E_w} \langle \chi , w\alpha^\vee \rangle \overline{\Sbt}_{w s_\alpha}.
\end{equation}
\end{assertionlist}
\end{theorem}
For each $\chi\in X^*(T)$ and $w\in W$, we denote by $f_{w,\chi}$ a nonzero element in the space $H^0\left(\Sch_w,\Vcal_{\Sch}(\chi,-w^{-1} \chi) \right)$. By Theorem \ref{brion}, the section $f_{w,\chi}$ extends to the Zariski closure $\overline{\Sch}_w$ if and only if $\langle \chi, w\alpha^\vee \rangle \leq 0$ for all $\alpha\in E_w$. By \cite[Lemma 3.1.1 (b)]{Goldring-Koskivirta-Strata-Hasse}, the pullback of $\Vcal_{\Sbt}(\chi,\eta)$ via the morphism $\psi\colon \GF^\mu\to \Sbt$ defined in \eqref{psi-eq} is:
\begin{equation}
\psi^*(\Vcal_{\Sbt}(\chi,\nu))=\Vcal_{\flag}(\chi + p w_{0,I}w_0\sigma^{-1}(\nu))    
\end{equation}
(note that \loccit contains a typo; it should be $\sigma^{-1}$ instead of $\sigma$). For $w\in W$, define 
\begin{equation}\label{hw-eq}
 h_w\colon X^*(T)\to X^*(T),  \quad \chi\mapsto -w\chi + p w_{0,I}w_0\sigma^{-1}(\chi).   
\end{equation}
For all $\chi\in X^*(T)$ and all $w\in W$, define
\begin{equation}
    \Ha_{w,\chi} \colonequals \psi^*(f_{w,-w\chi}).
\end{equation}
By construction, we have
\begin{equation}
    \Ha_{w,\chi} \in H^0(\Fcal_w, \Vcal_{\flag}(h_w(\chi))).
\end{equation}
Since $\psi$ is smooth, the multiplicity of $\Ha_{w,\chi}$ along $\Fcal_{ws_\alpha}$ is the same as the multiplicity of $f_{w,-w\chi}$ along $\Sbt_{ws_\alpha}$, namely $\langle \chi,\alpha^\vee\rangle$. We are particularly interested in sections which vanish exactly on the Zariski closure of a single stratum $\overline{\Fcal}_{ws_\alpha}\subset \overline{\Fcal}_w$ (where $\alpha \in E_w$). This kind of section is sometimes called a partial Hasse invariant. 

\begin{definition}\label{def-syst-PHI}
We say that $w\in W$ admits a separating system of partial Hasse invariants if the elements $\{\alpha^\vee \ | \ \alpha\in E_w \}$ are linearly independent in the $\QQ$-vector space $X_*(T)\otimes_\ZZ \QQ$.
\end{definition}

If $w$ admits a separating system of partial Hasse invariants, then each flag stratum $\overline{\Fcal}_{ws_\alpha}\subset \overline{\Fcal}_w$ ($\alpha\in E_w$) can be cut out by a partial Hasse invariant. Indeed, by linear independence, for each $\beta \in E_w$ we can find a character $\chi_\beta\in X^*(T)$ satisfying
\begin{equation}\label{chibeta}
    \begin{cases}
        \langle \chi_\beta, \beta^\vee\rangle >0 & \textrm{and} \\
        \langle \chi_\beta, \alpha^\vee\rangle =0 & \textrm{for all } \alpha\in E_w\setminus\{\beta\}.
    \end{cases}
\end{equation}
For this character $\chi_\beta$, the corresponding section $\Ha_{w,\chi_\beta}$ is a partial Hasse invariant which vanishes exactly on $\overline{\Fcal}_{ws_\beta}$ inside $\overline{\Fcal}_w$. The weight of $\Ha_{w,\chi_\beta}$ is
\begin{equation}\label{weight-Ha}
    \ha_{w,\chi_\beta} \colonequals h_w(\chi_\beta) = -w\chi_\beta + p w_{0,I}w_0\sigma^{-1}(\chi_\beta).
\end{equation}

\subsection{Regularity of strata}
Although we are mainly interested in proving results pertaining to automorphic forms in $H^0(S_K,\Vcal_I(\lambda))$, it is useful to consider sections on smaller flag strata or zip strata, or the Zariski closures thereof. This strategy was carried out in \cite{Goldring-Koskivirta-Strata-Hasse}, where exact sequences between cohomology groups of various strata played a key role to move Hecke eigen-systems from higher degree cohomology to degree $0$. For this reason, we introduce the following, somewhat cumbersome but useful notation and terminology. For each $w\in W$, denote by $X_{+,w}^*(T)$ the set of characters $\chi\in X^*(T)$ such that $\langle \chi,\alpha^\vee\rangle \geq 0$ for all $\alpha\in E_w$. We may call such characters "$w$-dominant". Define the cone of partial Hasse invariants of $w$ by:
\begin{equation}
    C_{\Hasse,w}\colonequals h_w(X_{+,w}^*(T))
\end{equation}
where $h_w$ is the map \eqref{hw-eq}. By Chevalley's formula (Theorem \ref{brion}), the section $\Ha_{w,\chi}$ extends to $\overline{\Fcal}_w$ if and only if $\chi\in X_{+,w}^*(T)$. Therefore, the set $C_{\Hasse,w}$ is the set of all $\lambda\in X^*(T)$ such that $\Vcal_{\flag}(\lambda)$ admits a nonzero section over $\overline{\Fcal}_w$ which arises by pullback from a section over $\overline{\Sbt}_w$. In general, there are many sections on $\overline{\Fcal}_w$ which do not arise by pullback in this fashion. To account for them, we define the following cone:
\begin{equation}
    C_{\flag, w} \colonequals \{\lambda\in X^*(T) \ | \ H^0(\overline{\Fcal}_w, \Vcal_{\flag}(\lambda))\neq 0\}.
\end{equation}
By the above discussion, it is clear that $C_{\Hasse,w}\subset C_{\flag, w}$. 

Consider now a $k$-scheme $X$ endowed with a surjective map $\zeta\colon X\to \GZip^\mu$ (for example $X=S_K$ or $X=S^{\Sigma}_K$). Let $\Flag(X)$ be the flag space of $X$ (diagram \eqref{zeta-flag}). Let $(\Flag(X)_w)_{w\in W}$ be the flag stratification on $\Flag(X)$ and write $\overline{\Flag}(X)_w$ for the Zariski closure of $\Flag(X)_w$ (endowed with the reduced structure). Define similarly:
\begin{equation}\label{CXw-eq}
C_{X, w} \colonequals \{\lambda\in X^*(T) \ | \ H^0(\overline{\Flag}(X)_w, \Vcal_{\flag}(\lambda))\neq 0\}.
\end{equation}
Since $\zeta_{\flag}$ is surjective, the pullback map via $\zeta_{\flag}$ is injective, so we get inclusions
\begin{equation}\label{inclu}
    C_{\Hasse,w}\subset C_{\flag,w}\subset C_{X, w}.
\end{equation}
Note that when $w=w_0$, the identification \eqref{identif-lambda} shows that $C_{\flag, w_0} = C_{\zip}$ and similarly $C_{X,w_0}=C_{X}$. For $w=w_0$, we simply write $C_{\Hasse,w_0}=C_{\Hasse}$. Concretely, $C_{\Hasse}$ is the set of non-negative linear combinations of the weights of the (global) partial Hasse invariants, similarly to those studied in \cite{Diamond-Kassaei} and \cite{Imai-Koskivirta-partial-Hasse}. In all cones defined above, we change the letter $C$ to a calligraphic $\Ccal$ to denote the saturation.

\begin{definition} \label{def-regular} Let $w\in W$.
\begin{definitionlist}
\item We say that the flag stratum $\Flag(X)_w$ is Hasse-regular if $\Ccal_{X,w} = \Ccal_{\Hasse, w} $.
\item We say that the flag stratum $\Flag(X)_w$ is flag-regular if $\Ccal_{X,w} = \Ccal_{\flag, w}$.
\end{definitionlist}
\end{definition}
In view of the inclusions \eqref{inclu}, it is clear that any Hasse-regular stratum is flag-regular. Conjecture \ref{cone-conj} asserts that the maximal stratum $\Flag(X)_{w_0}$ is flag-regular. In certain cases, $\Flag(X)_{w_0}$ is even Hasse-regular. For example, this is the case for Hilbert--Blumenthal modular varieties,  Picard surfaces, Siegel threefolds (\cite[Theorem D]{Goldring-Koskivirta-global-sections-compositio}). However, in general the stratum $\Flag(X)_{w_0}$ is not Hasse-regular because the inclusion $\Ccal_{\Hasse} \subset \Ccal_{\zip}$ is usually strict. A complete characterization of the cases when the equality $\Ccal_{\Hasse} = \Ccal_{\zip}$ holds is given in \cite[Theorem 4.3.1]{Imai-Koskivirta-zip-schubert}. Returning to the case of Hilbert--Blumenthal modular varieties (attached to a totally real field $\mathbf{F}$), the paper \cite{Goldring-Koskivirta-global-sections-compositio} also covers the case of smaller strata. There is an explicit and straight-forward criterion (see \loccit Theorem 4.2.3) on the element $w\in W$ for the Hasse-regularity of $S_{K,w}$. When $p$ is split in $\mathbf{F}$, all strata are Hasse-regular.

The flag stratum parametrized by the element $w_{\max}=w_{0,I}w_0$ plays a central role in our strategy. In the case of PEL unitary Shimura varieties of signature $(n-1,1)$ at a split prime, Goldring and the author showed that $\Flag(X)_{w_{\max}}$ is Hasse-regular (\cite{Goldring-Koskivirta-GS-cone}). Goldring pointed out to us that it is not Hasse-regular in general. When $P$ is defined over $\FF_p$, the cone $C_{\Hasse,w_{\max}}$ has a very simple form:
\begin{equation}
    C_{\Hasse,w_{\max}}=\{\lambda\in X^*(T) \ | \ \langle \lambda, \alpha^\vee\rangle \leq 0, \ \textrm{for all} \ \alpha\in \Phi^+\setminus \Phi^+_{L} \}.
\end{equation}
In the case of PEL unitary Shimura varieties of signature $(n-1,1)$ at a split prime, the Hasse-regularity of $\Flag(X)_{w_{\max}}$ was proved by Goldring and the author (\cite{Goldring-Koskivirta-GS-cone}) using a highly computational and tedious combinatorial argument. In the present paper concerned with the case of Siegel-type Shimura varieties, we will replace this argument by a much more conceptual approach involving the density of Hecke-orbits. This Hasse-regularity result will be the main ingredient in the proof of Conjecture \ref{conj-orb-ineq}.

\section{Vanishing result for Siegel automorphic forms}

In this section, we return to the the case of the Siegel-type Shimura variety $\overline{\Ascr}_{n,K}$ (section \ref{Siegel-type-sec}). we first show the Hasse-regularity of $\Flag(X)_{w_{\max}}$, and then use it to prove Conjecture \ref{conj-orb-ineq}.

\subsection{Main result}

Recall that we defined $\Ccal^{+,I}_{L-\Min} \colonequals \Ccal_{L-\Min}\cap X^*_{+,I}(T)$ (equation \eqref{LMinplusI}). It is often the case that many of the $W_{L}\rtimes \Gal(\FF_{p^2}/\FF_p)$-orbits $\Ocal\subset \Phi^+\setminus \Phi^+_{L}$ and subsets $S\subset \Ocal$ appearing in \eqref{orb-ineq} will contribute trivially when we intersect with the $I$-dominant characters $X^*_{+,I}(T)$. Let us examine the case of $G=\GSp_{2n,\FF_p}$. We retain the notation of \S\ref{Siegel-type-sec}. Write $T$ for the diagonal torus of $G$. It is generated by the subgroup $Z\subset T$ of nonzero scalar matrices and the subtorus $T_0\colonequals T\cap \Sp_{2n}$. Explicitly, $T_0$ is given by
\begin{equation}
    T_0 = \{ \diag(t_1,\dots, t_n,t_n^{-1},\dots,t_1^{-1}) \ | \ t_i\in \GG_{\mathrm{m}} \}.
\end{equation}
Identify $X^*(T_0)=\ZZ^n$ such that $(a_1,\dots,a_n)$ corresponds to the character
\begin{equation}\label{charT-Sp}
    \diag(t_1,\dots, t_n,t_n^{-1},\dots,t_1^{-1})\mapsto \prod_{i=1}^n t_i^{a_i}.
\end{equation}
We also identify $X^*(Z)=\ZZ$, such that $b\in \ZZ$ corresponds to the character $\diag(t,\dots,t)\mapsto t^b$. The natural map $X^*(T)\to X^*(T_0)\times X^*(Z)=\ZZ^n\times \ZZ$, $\chi\mapsto (\chi|_{T_0},\chi|_Z)$ is injective, and identifies $X^*(T)$ with the set
\begin{equation}\label{XT-GSp}
    X^*(T) = \left\{(a_1,\dots,a_n,b)\in \ZZ^{n+1} \ \relmiddle| \ \sum_{i=1}^n a_i \equiv b \pmod 2 \right\}.
\end{equation}
Write $e_1,\dots, e_n$ for the canonical basis of $\ZZ^n$, and write again $e_i$ for the element $(e_i,0)\in \ZZ^{n+1}$. The positive roots of $G$ are
\begin{equation}
    \Phi^+\colonequals \{e_i\pm e_j \ | \ 1\leq i<j \leq n \} \sqcup \{2e_i \ | \ 1\leq i\leq n \}.
\end{equation}
Hence, the set $\Phi^+\setminus \Phi^+_{L}$ consists of exactly two $W_L$-orbits, namely
\begin{equation}
     \Ocal_1=\{2e_i \ | \ 1\leq i\leq n \} \quad \textrm{and} \quad \Ocal_2=\{e_i+e_j \ | \ 1\leq i<j\leq n \}.
\end{equation}
One can show that the inequalities \eqref{orb-ineq} corresponding to subsets $S\subset \Ocal_2$ become redundant when intersecting with  $X^*_{+,I}(T)$, so we may restrict to subsets of $\Ocal_1$. Similarly, only subsets of the form $S=\{2e_1,2e_2,\dots, 2e_j\}$ (for $j=1,\dots, n$) are relevant. Therefore, the cone $\Ccal^{+,I}_{L-\Min}$ can be defined inside $X^*_{+,I}(T)$ by the following inequalities:
\begin{equation}\label{ineq-thm}
 \Ccal^{+,I}_{L-\Min} = \left\{ (a_1,\dots,a_n,b)\in X^*_{+,I}(T) \ \relmiddle| \ \sum_{i=1}^j a_i + \frac{1}{p} \sum_{i=j+1}^{n} a_i \leq 0 \ \textrm{ for }j=1,\dots, n\right\}.
\end{equation}
Actually, one can show that the inequality corresponding to $j=n$ is also redundant and can be removed. The following theorem is our main result on the topic of cohomology vanishing for Siegel automorphic forms. We explain the proof in section \ref{sec-proof}.

\begin{theorem} \label{main-thm}
One has $\Ccal(\overline{\FF}_p)\subset \Ccal^{+,I}_{L-\Min}$. In other words, if the character $\lambda$ fails to satisfy any of the above inequalities \eqref{ineq-thm}, then $H^0(\overline{\Ascr}_{n,K}, \Vcal_I(\lambda))=0$.
\end{theorem}
In the case $n=3$, it was shown for $p\geq 5$ in \cite{Goldring-Koskivirta-divisibility} that the cone $\Ccal(\overline{\FF}_p)$ coincides with the zip cone $\Ccal_{\zip}$ (Conjecture \ref{cone-conj}) and is given as follows:
\begin{equation}
    \Ccal(\overline{\FF}_p) = \left\{ (a_1,a_2,a_3,b)\in X^*_{+,I}(T) \ \relmiddle| \ \parbox{4cm}{ $p^2 a_1 + a_2 + pa_3 \leq 0 \\
    pa_1 + p^2a_2 + a_3 \leq 0$ } \right\}.
\end{equation}
This cone is only slightly smaller than the cone $\Ccal^{+,I}_{L-\Min}$, which is given by
\begin{equation}
\Ccal^{+,I}_{L-\Min}  = \left\{ (a_1,a_2,a_3,b)\in X^*_{+,I}(T) \ \relmiddle| \ \parbox{4cm}{ $p a_1 + a_2 + a_3 \leq 0 \\
pa_1 + pa_2 + a_3 \leq 0$ } \right\}
\end{equation}
(see \cite[Figure 1]{Goldring-Koskivirta-divisibility} for a visual plot of these cones). This shows that the upper bound for $\Ccal(\overline{\FF}_p)$ provided by Theorem \ref{main-thm} is a good approximation. However, it would be interesting to determine the precise inequalities defining $\Ccal(\overline{\FF}_p)$ for general $n$.

\subsection{\texorpdfstring{Bruhat order of $W$}{}}

Recall that the Weyl group $W$ of $G=\GSp_{2n}$ can be identified with the set of permutations $w\in \Sfr_{2n}$ satisfying $w(i)+w(2n+1-i)=2n+1$ for all $1\leq i \leq 2n$. For $w\in W$, put
\begin{align}
\Mcal(w)&\colonequals \{(i,j) \ | \  1\leq i<j\leq n \ \textrm{and} \ w(i)>w(j) \} \\
\Ncal(w)&\colonequals \{(i,j) \ | \  1\leq i\leq j\leq n \ \textrm{and} \ w(i)+w(j)>2n+1 \}.
\end{align}
Write $M(w)$, $N(w)$ for the number of elements of $\Mcal(w)$ and $\Ncal(w)$ respectively. The length $\ell(w)$ of $w\in W$ is given by
\begin{equation}\label{length-eq}
    \ell(w)=M(w)+N(w).
\end{equation}
For $w\in W$ and $1\leq i,j\leq 2n$, define
\begin{equation}
    r_w(i,j)\colonequals |\{1\leq k\leq i \ | \ w(k) \leq j\}|.
\end{equation}
For two elements $w_1,w_2\in W$, one has an equivalence
\begin{equation}\label{bruhat-eq}
    w_1\leq w_2 \ \Longleftrightarrow \ r_{w_1}(i,j)\geq r_{w_2}(i,j) \quad \textrm{for all } 1\leq i,j\leq 2n.
\end{equation}
For any permutation $\tau\in \Sfr_{2n}$, write $M_\tau\in \GL_{2n}$ for the permutation matrix of $\tau$. We say that a pair $(i,j)$ is admissible for $\tau$ if it satisfies the following conditions:
\begin{bulletlist}
    \item $1\leq i<j\leq 2n$
    \item $\tau(i)>\tau(j)$
    \item There is no $i<k<j$ such that $\tau(j)<\tau(k)<\tau(i)$.
\end{bulletlist}
In other words, $(i,j)$ is admissible for $\tau$ if and only if the submatrix of $M_\tau$ with corners $(i,\tau(i))$ and $(j,\tau(j))$ has only zero coefficients, except for these two corners. Now, we let $w\in W$ and define three sets $\Ecal^1_w$, $\Ecal_w^2$, $\Ecal_w^3$ as follows:
\begin{align}
    \Ecal_w^1 &\colonequals \{(i,j) \ | \ \textrm{admissible for $w$, and } 1\leq i<j\leq n \} \\
    \Ecal_w^2 &\colonequals \{(i,j) \ | \ \textrm{admissible for $w$, }i\leq n<j  \textrm{ and } w(i),w(j)\leq n \} \\
    \Ecal_w^3 &\colonequals \{(i,j) \ | \ \textrm{admissible for $w$, } i\leq n 
 \textrm{ and } j=2n+1-i \}.
\end{align}
It is easy to see that $\Ecal_w^1$, $\Ecal_w^2$, $\Ecal_w^3$ are pairwise disjoint. We set
\begin{equation}
    \Ecal_w=\Ecal_w^1\sqcup \Ecal_w^2 \sqcup \Ecal_w^3.
\end{equation}
For $(i,j)\in \Ecal_w$, define a positive root $\gamma(i,j)\in \Phi^+$ as follows:
\begin{equation}
    \gamma(i,j)\colonequals \begin{cases}
        e_i-e_j & \textrm{if } (i,j)\in \Ecal_w^1 \\
         e_i+e_{2n+1-j} & \textrm{if }  (i,j)\in \Ecal_w^2 \\
          2e_i & \textrm{if }  (i,j)\in \Ecal_w^3.
    \end{cases}
\end{equation}
Recall the set $E_w\subset \Phi^+$ defined in \eqref{def-Ew}. One has the following lemma:

\begin{lemma}\label{lemma-gamma}
    The map $\gamma$ is a bijection of $\Ecal_w$ onto the set $E_w$.
\end{lemma}

\begin{proof}
We first show that when $(i_0,j_0)\in \Ecal_w$, the positive root $\alpha=\gamma(i_0,j_0)$ lies in $E_w$. By the characterization of the Bruhat order (see \eqref{bruhat-eq}), we clearly have $w s_\alpha < w$. Next, we need to show that $\ell(w s_\alpha)=\ell(w)-1$. We consider the following cases:
\begin{bulletlist}
\item The case $ (i_0,j_0)\in \Ecal_w^1$: We construct a bijection $\epsilon\colon \Mcal(w)\setminus\{(i_0,j_0)\}\to \Mcal(ws_\alpha)$. Let $(i,j)\in \Mcal(w)\setminus\{(i_0,j_0)\}$.
\begin{equation}
    \begin{cases}
\textrm{if } i<i_0 \textrm{ and } w(i_0)>w(i)>w(j_0), & \ \textrm{we set } \epsilon(i,j_0)=(i,i_0) \\
\textrm{if } j_0<j \textrm{ and } w(i_0)>w(j)>w(j_0), & \ \textrm{we set } \epsilon(i_0,j)=(j_0,j).
\end{cases}
\end{equation}
In all other cases, set $\epsilon(i,j)=(i,j)$. Then on checks easily that $\epsilon$ is a bijection $\epsilon\colon \Mcal(w)\setminus\{(i_0,j_0)\}\to \Mcal(ws_\alpha)$. Therefore, we have $M(ws_\alpha)=M(w)-1$. On the other hand, if two permutations $\tau$, $\tau'$ satisfy $\{\tau(1), \dots, \tau(n)\}=\{\tau'(1),\dots, \tau'(n)\}$, then one clearly has $N(\tau)=N(\tau')$. Hence we deduce $N(ws_\alpha)=N(w)$ and thus $\ell(ws_\alpha)=\ell(w)-1$.
\item The case $(i_0,j_0)\in \Ecal_w^2$: In this case, one has $(w(j_0),w(i_0))\in \Ecal_{w^{-1}}^1$. We have $ws_\alpha = s_{w\cdot \alpha} w$. By the first case above, we deduce 
\begin{equation}
    \ell(ws_\alpha)=\ell(s_{w\cdot \alpha} w)=\ell(w^{-1}s_{w\cdot \alpha})=\ell(w^{-1})-1=\ell(w)-1.
\end{equation}
\item The case $(i_0,j_0)\in \Ecal_w^3$: Denote by $A$ the set of pairs $(i,i_0)$ satisfying $1\leq i <i_0$ and $w(j_0)<w(i)<w(i_0)$. One shows easily the following:
\begin{align}
    \Mcal(ws_\alpha)&=\Mcal(w)\sqcup A \\
    \Ncal(ws_\alpha)&=\Ncal(w)\setminus (A\cup \{(i_0,i_0)\})
\end{align}
It follows that $\ell(ws_\alpha)=\ell(w)-1$. This terminates the proof of $\gamma(\Ecal_w)\subset E_w$.
\end{bulletlist}
\bigskip

Finally, we show that any $\alpha\in E_w$ is of the form $\gamma(i_0,j_0)$ for $(i_0,j_0)\in \Ecal_w$. We consider the case $\alpha=e_{i_0}-e_{j_0}$ for $1\leq i_0<j_0\leq n$. Since $ws_\alpha < w$, it is clear that $w(i_0)>w(j_0)$. We check that $(i_0,j_0)$ is admissible for $w$. If $i_0<k<j_0$ satisfies $w(i_0)>w(k)>w(j_0)$, then define $\beta\colonequals e_{i_0}-e_k$. Using \eqref{bruhat-eq}, one sees immediately that $w>ws_\beta > ws_\alpha$. This contradicts the assumption $\ell(ws_\alpha)=\ell(w)-1$. Hence $(i_0,j_0)$ is admissible for $w$. In the case $\alpha=e_{i_0}+e_{j_0}$ (resp. $\alpha=2e_{i_0}$), a similar argument applies to show that $(i_0,j_0)$ (resp. $(i_0,2n+1-i_0)$) is admissible. We deduce the result.
\end{proof}

\subsection{Auxilliary sequence}\label{subsec-auxil}

Using the same strategy as in \cite{Goldring-Koskivirta-GS-cone}, we define a descending path in $W$ from the longest element $w_0$ of $W$ to $w_{\max}=w_{0,I}w_0$ (the longest element of ${}^I W$) satisfying certain requirements. Specifically, we construct a sequence of elements $\tau_1,\tau_2,\dots,\tau_N$ satisfying the following conditions:
\begin{definitionlist}
\item $\tau_1=w_0$ and $\tau_N=w_{\max}$.
\item $\tau_1>\dots > \tau_N$ and $\ell(\tau_{i+1})=\ell(\tau_{i})-1$ for each $i=1,\dots,N-1$.
\item Each $\tau_i$ admits a separating system of partial Hasse invariants (Definition \ref{def-syst-PHI}).
\end{definitionlist}

\bigskip

For $1\leq d\leq n$, let $\Lambda_d\in W$ be the element corresponding to the following permutation:
\begin{equation}
    \Lambda_d \colonequals 
   \left( \begin{array}{ccc|ccc|ccc}
        &&& & & &1 & &  \\
        &&& & & & &\ddots &  \\
        &&& & & & & & 1 \\ \hline
        &&& & & 1 & & &  \\
        &&& & \iddots& & & &  \\
        &&& 1& & & & & \\ \hline 
        1&&& & & & & & \\ 
        &\ddots&& & & & & & \\ 
        &&1& & & & & & \\ 
    \end{array}\right)
\end{equation}
where the upper right block and the lower left block have size $d\times d$ and the middle block has size $(2n-2d)\times (2n-2d)$. In particular, $\Lambda_1=w_0$ is the longest element of $W$ and $\Lambda_n=w_{0,I}w_0=w_{\max}$ is the longest element of ${}^I W$. For each $1\leq d\leq n-1$, we construct a path from $\Lambda_d$ to $\Lambda_{d+1}$ as follows. Define
\begin{align}
    \tau_d^{(0)} &= \Lambda_d  \\
    \tau_d^{(1)} &= \tau_d^{(0)}  s_{\alpha_1} \quad  \textrm{with} \ \alpha_1=e_{1}-e_{d+1} \\
 \tau_d^{(2)} &= \tau_d^{(1)} s_{\alpha_2} \quad \textrm{with} \ \alpha_2=e_{2}-e_{d+1} \\
 &\vdots & \\
 \tau_d^{(d-1)} &= \tau_d^{(d-2)} s_{\alpha_{d-1}} \quad \textrm{with} \ \alpha_{d-1} =e_{d-1}-e_{d+1} \\
 \tau_d^{(d)} &= \tau_d^{(d-1)} s_{\alpha_d} \quad \ \ \textrm{with} \ \alpha_d=e_{d}-e_{d+1}     
\end{align}
At each step, the coefficient 1 in the $d+1$-th row of the matrix moves down by one. Therefore, after $d$ steps, we have $\tau_d^{(d)} = \Lambda_{d+1}$. It is easy to check that $\ell(\tau_d^{(i+1)})=\ell(\tau_d^{(i)})-1$. By this construction, we obtain a path in $W$ from $\Lambda_d$ to $\Lambda_{d+1}$. By concatenating these paths, we obtain a path from $\Lambda_1=w_0$ to $\Lambda_n=w_{\max}$.

\begin{proposition}\label{prop-sep-syst}
For all $1\leq d\leq n-1$ and all $0\leq i \leq d-1$, the element $\tau_d^{(i)}$ admits a separating system of partial Hasse invariants.
\end{proposition}

\begin{proof}
We use Lemma \ref{lemma-gamma}. For each $1\leq d\leq n-2$ and each $0\leq i \leq d-1$, the element $\tau_d^{(i)}$ has exactly $n$ lower neighbors. By the shape of the permutation matrix of $\tau_d^{(i)}$, we see that the set $E_{\tau_d^{(i)}}$ is made of five different parts, namely:
\begin{align}
    E_{\tau_d^{(i)}}=\{e_k-e_{d+2} \ | \ 1\leq k \leq i\} \ \cup \ \{e_{d+1}-e_{d+2}\} \ \cup \ \{e_k-e_{d+1} \ | \ i+1\leq k \leq d\} \\
    \cup \ \{e_{k}-e_{k+1} \ | \ d+2\leq k \leq n-1\} \ \cup \ \{2e_n\}.
\end{align}
It is straight-forward to check that the above elements are linearly independent. Finally, in the case $d=n-1$, the number of lower neighbors of $\tau_{n-1}^{(i)}$ is $n-i-1$ for all $0\leq i \leq n-2$. Specifically, we have
\begin{equation}
    E_{\tau_{n-1}^{(i)}} = \{e_k-e_n \ | \ i+1\leq k \leq n-1\}.
\end{equation}
Again, the above elements are linearly independent. This terminates the proof.
\end{proof}

In particular, there exists a partial Hasse invariant on the flag stratum parametrized by $\tau_{d}^{(i)}$ whose vanishing locus is exactly the Zariski closure of the stratum parametrized by $\tau_{d}^{(i+1)}$, for each $1\leq d\leq n-1$ and each $0\leq i\leq d-1$. We first examine the case $1\leq d\leq n-2$. In this case, the elements of $E_{\tau_{d}^{(i)}}$ form a basis of $\ZZ^n$, so there is a unique such partial Hasse invariant (up to multiple). To determine the weight of this partial Hasse invariant, recall that $\tau_{d}^{(i+1)}=\tau_{d}^{(i)} s_{\alpha}$ with $\alpha=e_{i+1}-e_{d+1}$. We consider the character 
\begin{equation}
    \chi_d^{(i)}\colonequals e_{i+1}.
\end{equation}
By the proof of Proposition \ref{prop-sep-syst}, $\chi_d^{(i)}$ is orthogonal to all elements of $E_{\tau_{d}^{(i)}}$ different from $\alpha=e_{i+1}-e_{d+1}$, and satisfies $\langle \chi_d^{(i)}, \alpha^\vee \rangle =1$. By \eqref{weight-Ha} the corresponding partial Hasse invariant $\Ha_{\tau_{d}^{(i)},\chi_d^{(i)}}$ has weight
\begin{align}
    \ha_{\tau_{d}^{(i)},\chi_d^{(i)}} &=  - \tau_{d}^{(i)} \chi_d^{(i)} + p w_{0,I}w_0 \chi_d^{(i)} \\
    & = e_{d-i+1} - p e_{n-i}. \label{eq-ha-e}
\end{align}
We denote the above character simply by $\ha_d^{(i)}$.

\begin{lemma}\label{lemma-ha-Lmin}
For each $1\leq d\leq n-1$ and each $0\leq i\leq d-1$, the weight $\ha_d^{(i)}$ lies in the $L$-minimal cone \eqref{ineq-thm}.
\end{lemma}

\begin{proof}
This follows immediately from \eqref{eq-ha-e}, noting that $d-i+1\leq n-i$.
\end{proof}

\subsection{Hasse-regularity}\label{sec-Hasse-reg}

In this section, we prove that the flag stratum $\Flag(X)_{w_{\max}}$ is Hasse-regular in the case $X=\overline{\Ascr}_{n,K}$. For $w\in W$, we write $C_{K,w}$ for the cone $C_{X,w}$ defined in \eqref{CXw-eq}, and $\Ccal_{K,w}$ for its saturation. Recall that we parametrize characters in $X^*(T)$ by $n+1$-tuples $(a_1,\dots,a_n,b)$ satisfying $\sum_{i=1}^n a_i \equiv b \pmod{2}$ (see \eqref{XT-GSp}).

\begin{theorem}\label{thm-Hasse-reg-wmax}
The flag stratum $\Flag(\overline{\Ascr}_{n,K})_{w_{\max}}$ is Hasse-regular. Explicitly, one has
\begin{equation}
    \Ccal_{K,w_{\max}} = \Ccal_{\Hasse,w_{\max}} = \{(a_1,\dots,a_n,b)\in X^*(T) \ | \ a_i\leq 0, \ i=1,\dots,n \}.
\end{equation}
\end{theorem}

The above theorem has also consequences for automorphic forms over $\overline{\FF}_p$. Indeed, let $f\in H^0(\overline{\Ascr}_{n,K},\Vcal_I(\lambda))$ be an automorphic form and assume that $\lambda=(a_1,\dots,a_n,b)$ satisfies $a_j>0$ for some $1\leq j \leq n$. We may view $f$ as a section of $\Vcal_{\flag}(\lambda)$ on the flag space of $\overline{\Ascr}_{n,K}$ using the identification \eqref{identif-lambda}. Then, the above theorem implies that $f$ vanishes identically on the flag stratum $\Flag(\overline{\Ascr}_{n,K})_{w_{\max}}$. Write simply $\BB(\lambda)$ for the stable base-locus of $\Vcal_{\flag}(\lambda)$ (since the level $K$ is fixed).

\begin{corollary}
Let $\lambda=(a_1,\dots, a_n,b)\in X^*(T)$ and assume that $\lambda \notin \Ccal_{\GS}$. Then we have
\begin{equation}
    \overline{\Flag}(\overline{\Ascr}_{n,K})_{w_{\max}} \subset \BB(\lambda).
\end{equation}
\end{corollary}

\begin{proof}
The condition $\lambda \notin \Ccal_{\GS}$ means precisely that some coordinate $a_i$ is positive. The above argument then shows that any element of $H^0(\Flag(\overline{\Ascr}_{n,K}),\Vcal_{\flag}(\lambda))$ vanishes on the flag stratum $\Flag(\overline{\Ascr}_{n,K})_{w_{\max}}$. Since the same argument applies to positive multiples of $\lambda$, we deduce the result.
\end{proof}

To prove Theorem \ref{thm-Hasse-reg-wmax}, we use an embedding from a Hilbert--Blumenthal Shimura variety into $\overline{\Ascr}_{n,K}$. Note that the cones $\Ccal_{w}$ are independent of the choice of the level $K$ (by the same formal argument as \cite[Corollary 1.5.3]{Koskivirta-automforms-GZip}). Therefore, we may change $K$ freely if necessary. Choose a totally real field $\mathbf{F}$ such that $p$ splits in $\mathbf{F}$, and consider the associated Hilbert--Blumenthal Shimura variety as in section \ref{hb-var-sec}. Choose open compact subgroups $K'^p\subset \mathbf{H}(\AA_f^p)$ and $K^p\subset \mathbf{G}(\AA_f^p)$ such that $u(K'^p)\subset K^p$, where $u\colon \mathbf{H}\to \mathbf{G}$ is the embedding \eqref{u-embedQ}. We adopt the same notation as in section \ref{sec-GZip-HT} and diagram \eqref{diag-HG}, namely we write $X_H\colonequals \overline{\Hscr}_{\mathbf{F},K'}$ and $X_G\colonequals \overline{\Ascr}_{n,K}$. Write also $\Flag_G$ for the flag space of $X_G$.

\begin{proof}[Proof of Theorem \eqref{thm-Hasse-reg-wmax}]
We first note that the formation of the line bundles $\Vcal_{\flag}(\lambda)$ is functorial. To explain this, we briefly consider the more general setting of diagram \eqref{diag-flag-zip}, where $H,G$ are arbitrary connected reductive $\FF_p$-groups endowed with cocharacters $\mu_H,\mu_G$ respectively and $f\colon H\to G$ is a compatible embedding. Write $B_H$, $B_G$ for the Borel subgroups as in \S\ref{sec-functo-flag} and choose maximal tori $T_H\subset B_H$ and $T_G\subset B_G$ such that $f(T_H)\subset T_G$. It is clear by construction that for any $\lambda\in X^*(T_G)$, one has
\begin{equation}\label{Vflag-pb}
     f_{\flag}^*(\Vcal_{\flag}(\lambda)) = \Vcal_{\flag}(f^*(\lambda))
\end{equation}
where $f^*(\lambda)$ is the character $\lambda\circ f\colon T_H\to \GG_{\mathrm{m}}$. We now return to the present setting. The embedding $u\colon H\to G$ induces an isomorphism $u\colon T_H\to T$ of tori. For $\lambda=(a_1,\dots, a_n,b)\in X^*(T)$, the character $u^*(\lambda)$ is given by
\begin{equation}
\left(
\left(\begin{matrix}
    x t_1 & \\ & x t_1^{-1}
\end{matrix} \right), \dots, \left(\begin{matrix}
    x t_n & \\ & x t_n^{-1}
\end{matrix} \right)
\right)\mapsto x^b \prod_{i=1}^n t_i^{a_i} .
\end{equation}
Let $\lambda=(a_1,\dots, a_n,b)\in X^*(T)$ and assume that $a_j>0$ for some $j\in \{1,\dots,n\}$. We want to show that $H^0(\overline{\Flag}_{G,w_{\max}},\Vcal_{\flag}(\lambda))=0$, where $\overline{\Flag}_{G,w_{\max}}$ is the Zariski closure of $\Flag_{G,w_{\max}}$ endowed with the reduced structure. Let $s\in H^0(\overline{\Flag}_{G,w_{\max}},\Vcal_{\flag}(\lambda))$ be a section. By Proposition \ref{uflag-image}, the $\mu_H$-ordinary locus (since $p$ splits in $\mathbf{F}$, this is simply the usual ordinary locus) of $X_H$ is mapped by $\widetilde{u}_{\flag}$ (see diagram \eqref{diag-HG}) into the flag stratum $\Flag_{G,w_{\max}}$. Thus, we obtain a map
\begin{equation}\label{map-uflag-bar}
    \widetilde{u}_{\flag}\colon X_H \to \overline{\Flag}_{G,w_{\max}}.
\end{equation}
Consider the pullback $\widetilde{u}_{\flag}^*(s)$ via the map \eqref{map-uflag-bar}. By the above functoriality property \eqref{Vflag-pb}, this is a section of $\Vcal_{\flag}(u^*(\lambda))$ over $X_H$. By
\cite[Corollary 4.2.4]{Goldring-Koskivirta-global-sections-compositio} or \cite[Corollary 8.2]{Diamond-Kassaei}, we know that Conjecture \ref{cone-conj} holds for Hilbert--Blumenthal Shimura varieties. Since we are assuming that $p$ splits in $\mathbf{F}$, the cone of $X_H$ is simply given by the condition $a_i\leq 0$ for all $i=1,\dots, n$. Hence, since we assumed $a_j>0$ for some $j\in \{1,\dots,n\}$, we deduce:
\begin{equation}
    H^0(X_H, \Vcal_{\flag}(u^*(\lambda))) = 0.
\end{equation}
In particular, $\widetilde{u}_{\flag}^*(s)=0$. This implies that any section of $\Vcal_{\flag}(\lambda)$ over $\overline{\Flag}_{G,w_{\max}}$ must vanish at each point in the image of $\widetilde{u}_{\flag}$. Furthermore, for any $r\geq 1$ we have $\Vcal_{\flag}(\lambda)^{\otimes r}=\Vcal_{\flag}(r\lambda)$, so the same argument shows that any section of $\Vcal_{\flag}(\lambda)^{\otimes r}$ over $\overline{\Flag}_{G,w_{\max}}$ is zero on the image of $\widetilde{u}_{\flag}$. In particular, the stable base locus of $\lambda$ on the stratum
$\overline{\Flag}_{G,w_{\max}}$ satisfies
\begin{equation}
    \BB_{K^p,w_{\max}}(\lambda) \neq \emptyset.
\end{equation}
By Corollary \ref{cor-sbl-zero}, we obtain
\begin{equation}
    H^0(\overline{\Flag}_{G,w_{\max}}, \Vcal_{\flag}(\lambda)^{\otimes r}) = 0, \quad \textrm{for all} \ r\geq 1.
\end{equation}
This terminates the proof of Theorem \ref{thm-Hasse-reg-wmax}.
\end{proof}

\subsection{Proof of the main result} \label{sec-proof}
In this final section, we prove Theorem \ref{main-thm}. This part of the argument is similar to the one used in \cite{Goldring-Koskivirta-GS-cone}. Assume that $w\in W$ admits a separating system of partial Hasse invariants (Definition \ref{def-syst-PHI}). Let $\beta\in E_w$ and let $\chi_\beta\in X^*(T)$ be a character satisfying
\begin{equation}
    \begin{cases}
        \langle \chi_\beta, \beta^\vee\rangle >0 & \textrm{and} \\
        \langle \chi_\beta, \alpha^\vee\rangle =0 & \textrm{for all } \alpha\in E_w\setminus\{\beta\}.
    \end{cases}
\end{equation}
as in \eqref{chibeta}. Write $m=\langle \chi_\beta, \beta^\vee\rangle$. Then, the corresponding section $\Ha_{w,\chi_\beta}$ vanishes exactly on $\overline{\Fcal}_{ws_\beta}$, with multiplicity $m$. Write $\ha_{w,\chi_\beta}$ for the weight of $\Ha_{w,\chi_\beta}$ (see \eqref{weight-Ha}). Let $s\in H^0(\overline{\Flag}_{G,w},\Vcal_{\flag}(\lambda))$ and assume that $s$ vanishes with multiplicity $d\geq 0$ along $\overline{\Flag}_{G,ws_\beta})$ (since Zariski closures of flag strata are normal, it makes sense to talk about multiplicities). If $d=0$, then the restriction of $s$ to $\overline{\Flag}_{G,ws_\beta}$ is nonzero. If $d>0$, then $s^m$ can be written as $s^m=(\Ha_{w,\chi_\beta})^d s'$ where $s'$ is not identically zero on $\overline{\Flag}_{G,ws_\beta}$. Hence the weight of $s'$ lies in $C_{K,ws_\beta}$, and we deduce $\lambda\in C_{K,ws_\beta} + \ZZ_{\geq 0} \ha_{w,\chi_\beta}$. We have shown:
\begin{equation}\label{CK-inclu-rec}
C_{K,w} \subset C_{K,ws_\beta} + \ZZ_{\geq 0} \ha_{w,\chi_\beta}.
\end{equation}
Therefore, we can propagate information about a particular stratum to the strata above it, provided that there exists a separating system of partial Hasse invariants.

\begin{proof}[Proof of Theorem \ref{main-thm}]
By Theorem \ref{thm-Hasse-reg-wmax}, the cone $\Ccal_{K,w_{\max}}$ is given by the condition $a_i\leq 0$ for all $i=1,\dots,n$. Therefore, it is clear that $\Ccal_{K,w_{\max}}$ is contained in $\Ccal_{L-\Min}$. Consider the sequence 
\begin{equation}
 (\tau^{(i)}_d)_{\substack{1\leq d\leq n-1\\ 0\leq i\leq d-1}}   
\end{equation}
constructed in section \ref{subsec-auxil}. By Proposition \ref{prop-sep-syst}, each element in this sequence admits a separating system of partial Hasse invariants. Furthermore, Lemma \ref{lemma-ha-Lmin} shows that for each element of the sequence, the weight of the partial Hasse invariant $\ha_d^{(i)}$ (for $1\leq d\leq n-1$ and $0\leq i \leq d-1$) lies in $\Ccal_{L-\Min}$. By using the inclusion \eqref{CK-inclu-rec} recursively, we deduce $C_{K,w_0} \subset \Ccal_{L-\Min}$. Since $C_{K,w_0}=C_K(\overline{\FF}_p)$, we deduce that $C_K(\overline{\FF}_p)\subset \Ccal_{L-\Min}$, hence also $\Ccal(\overline{\FF}_p)\subset \Ccal_{L-\Min}$. Since $\Ccal(\overline{\FF}_p)$ is always contained in the $I$-dominant cone $X^*_{+,I}(T)$, we find $\Ccal(\overline{\FF}_p)\subset \Ccal^{+,I}_{L-\Min}$.
\end{proof}

\bibliographystyle{test}
\bibliography{biblio_overleaf}

\noindent
Jean-Stefan Koskivirta\\
Department of Mathematics, Faculty of Science, Saitama University, 
255 Shimo-Okubo, Sakura-ku, Saitama City, Saitama 338-8570, Japan \\
jeanstefan.koskivirta@gmail.com

\end{document}